\documentclass[12pt,american]{article}
\usepackage[T1]{fontenc}
\usepackage[utf8]{inputenc}
\usepackage{amsmath}
\usepackage{amsthm}
\usepackage{amssymb}
\usepackage{geometry}
\usepackage{setspace}
\usepackage[authoryear]{natbib}
\makeatletter
\newcommand{\lyxaddress}[1]{
	\par {\raggedright #1
	\vspace{1.4em}
	\noindent\par}
}
\theoremstyle{plain}
\newtheorem{thm}{\protect\theoremname}
\theoremstyle{definition}
\newtheorem{defn}[thm]{\protect\definitionname}
\theoremstyle{plain}
\newtheorem{cor}[thm]{\protect\corollaryname}
\theoremstyle{plain}
\newtheorem{lem}[thm]{\protect\lemmaname}
\theoremstyle{remark}
\newtheorem{rem}[thm]{\protect\remarkname}
\theoremstyle{plain}
\newtheorem{prop}[thm]{\protect\propositionname}

\usepackage{babel}

\makeatother

\providecommand{\corollaryname}{Corollary}
\providecommand{\definitionname}{Definition}
\providecommand{\lemmaname}{Lemma}
\providecommand{\propositionname}{Proposition}
\providecommand{\remarkname}{Remark}
\providecommand{\theoremname}{Theorem}

\begin{document}
\title{On the identifiability of Dirichlet mixture models}
\author{Hien Duy Nguyen$^{1}$ and Mayetri Gupta$^{2}$}
\maketitle

\lyxaddress{$^{1}$School of Computing, Engineering and Mathematical Sciences,
La Trobe University\\
Bundoora, VIC 3086, Australia\\
Institute of Mathematics for Industry, Kyushu University\\
Fukuoka 819-0395, Japan\\
 \texttt{H.Nguyen5@latrobe.edu.au}\\[0.8em]$^{2}$School of Mathematics
and Statistics, University of Glasgow\\
Glasgow G12 8SQ, UK\\
\texttt{Mayetri.Gupta@glasgow.ac.uk}}
\begin{abstract}
We study identifiability of finite mixtures of Dirichlet distributions
on the interior of the simplex. We first prove a shift identity showing
that every Dirichlet density can be written as a mixture of $J$ shifted
Dirichlet densities, where $J-1$ is the dimension of the simplex
support, which yields non-identifiability on the full parameter space.
We then show that identifiability is recovered on a fixed-total parameter
slice and on restricted box-type regions. On the full parameter space,
we prove that any nontrivial linear relation among Dirichlet kernels
must involve at least $J$ coefficients sharing a common sign, and
deduce that mixtures with fewer than $J$ atoms are identifiable.
We further report direct non-identifiability implications for unrestricted finite mixtures of generalized
Dirichlet, Dirichlet-multinomial, fixed-topic-matrix latent Dirichlet allocation,
Beta-Liouville, and inverted Beta-Liouville models. 
\end{abstract}

\section{Introduction }

\label{sec:Introduction}

Dirichlet distributions are among the basic parametric families for
modeling random $J$-dimensional vectors on a probability simplex;
see, e.g., \citet[Ch. 49]{kotz2019continuous}. They arise naturally
in the analysis of compositional data and other simplex-valued observations.
When such data are generated by heterogeneous populations, finite
mixtures of Dirichlet distributions provide a natural model-based
clustering framework in the sense of \citet{McLachlanBasford1988}.
Dirichlet mixture models and related constructions have been used
in a range of applied settings; representative examples include \citet{bouguila2004unsupervised},
\citet{fan2012variational}, \citet{fan2016variational}, and \citet{pal2022clustering}.
In the special case $J=2$, Dirichlet mixtures reduce to mixtures
of beta distributions, which likewise arise in applications such as
bioinformatics, image analysis, and regression modeling; see \citet{houseman2008model},
\citet{ji2005applications}, \citet{ma2009beta}, and \citet{grun2012extended}.

For mixture-based clustering, identifiability is a basic structural
requirement. If identifiability fails, then the same mixture density
can admit genuinely different finite mixing distributions, so component-specific
interpretations and the induced clustering statements are not uniquely
determined. As usual, the relevant notion factors out label-switching
by identifying a finite mixture with its underlying finite mixing
measure; see \citet{teicher1963identifiability} and \citet[Sec. 3.1]{TitteringtonSmithMakov1985}.
The aim of this paper is to determine when finite mixtures of Dirichlet
distributions are identifiable and when they are not.

We now fix notation. Let $J\ge2$, and let
\[
\Delta_{J-1}^{\circ}=\left\{ x=(x_{1},\dots,x_{J})\in(0,1)^{J}:\sum_{j=1}^{J}x_{j}=1\right\} 
\]
be the interior of the $(J-1)$-simplex. Throughout, we view $\Delta_{J-1}^{\circ}$
through the coordinates $(x_{1},\dots,x_{J-1})$, with $x_{J}=1-\sum_{j=1}^{J-1}x_{j}$.
All densities and all ``almost everywhere'' statements on $\Delta_{J-1}^{\circ}$
are taken with respect to the corresponding $(J-1)$-dimensional Lebesgue
measure $dx_{1}\cdots dx_{J-1}$. For $\alpha=(\alpha_{1},\dots,\alpha_{J})\in\mathbb{R}_{>0}^{J}$,
write $\alpha_{+}=\sum_{j=1}^{J}\alpha_{j}$ and define the Dirichlet
density 
\begin{equation}
f_{\alpha}(x)=\frac{\Gamma(\alpha_{+})}{\prod_{j=1}^{J}\Gamma(\alpha_{j})}\prod_{j=1}^{J}x_{j}^{\alpha_{j}-1},\qquad x\in\Delta_{J-1}^{\circ}.\label{eq:Dirichlet_dens}
\end{equation}

A finite discrete mixing measure on $(0,\infty)^{J}$ is a probability
measure of the form 
\[
G=\sum_{k=1}^{K}\pi_{k}\,\delta_{\alpha^{(k)}},\qquad\pi_{k}>0,\ \sum_{k=1}^{K}\pi_{k}=1,
\]
where the atoms $\alpha^{(1)},\dots,\alpha^{(K)}$ are distinct and
$\delta_{\theta}$ denotes the Dirac measure at~$\theta\in\mathbb{R}_{>0}^{J}$.
The induced mixture density on $\Delta_{J-1}^{\circ}$ is 
\[
m_{G}(x)=\int f_{\alpha}(x)\,dG(\alpha)=\sum_{k=1}^{K}\pi_{k}f_{\alpha^{(k)}}(x).
\]
For $\Theta\subseteq\mathbb{R}_{>0}^{J}$, write 
\[
\mathcal{F}(\Theta)=\{m_{G}:G\text{ is a finite discrete mixing measure supported on }\Theta\},
\]
and write $\mathcal{F}=\mathcal{F}(\mathbb{R}_{>0}^{J})$ for brevity. 
\begin{defn}
Let $\Theta\subseteq\mathbb{R}_{>0}^{J}$. We say that the class of
finite Dirichlet mixtures supported on $\Theta$, $\mathcal{F}\left(\Theta\right)$,
is identifiable if whenever 
\[
m_{G}(x)=m_{G'}(x)\ \text{for almost every }x\in\Delta_{J-1}^{\circ}
\]
for two finite mixing measures $G,G'$ whose atoms lie in $\Theta$,
then $G=G'$ as measures. 
\end{defn}

For the one-dimensional simplex ($J=2$), it is well known that finite
mixtures of beta distributions are not identifiable; see, e.g., \citet{AhmadAlHussaini1982}
and \citet[Sec. 1.4.7]{chen2023statistical}. Our first result shows
that the same phenomenon persists on the full Dirichlet parameter
space $\Theta=\mathbb{R}_{>0}^{J}$. Indeed, if $e_{j}$ denotes the
$j$th standard basis vector of $\mathbb{R}^{J}$, then every Dirichlet
density satisfies the identity 
\[
f_{\alpha}(x)=\sum_{j=1}^{J}\frac{\alpha_{j}}{\alpha_{+}}f_{\alpha+e_{j}}(x).
\]
This expresses a single Dirichlet density as a mixture of $J$ distinct
Dirichlet densities and therefore yields non-identifiability of $\mathcal{F}$.

We then show that this obstruction disappears under a pair of restrictions
on the parameter space. First, for each fixed $A>0$, the class $\mathcal{F}(\Theta_{A})$
with 
\[
\Theta_{A}=\{\alpha\in\mathbb{R}_{>0}^{J}:\alpha_{+}=A\}
\]
is identifiable. Second, identifiability also holds on the full-dimensional
and open set 
\[
\Theta_{\mathrm{FD}}=\left\{ \alpha\in\mathbb{R}_{>0}^{J}:0<\alpha_{1},\dots,\alpha_{J-1}<1,\alpha_{J}>0\right\} .
\]
Finally, 
although the class $\mathcal{F}$ is not identifiable, we can show that the subclass
\[
\mathcal{F}_{J-1}=\left\{ m_{G}:G\text{ is a finite discrete mixing measure on }\mathbb{R}_{>0}^{J}\text{ with at most }J-1\text{ atoms}\right\} 
\]
is identifiable, in the sense that whenever $G$ and $G'$ are finite
mixing measures on $\mathbb{R}_{>0}^{J}$, each with at most $J-1$
atoms, and 
\[
m_{G}(x)=m_{G'}(x)\qquad\text{for almost every }x\in\Delta_{J-1}^{\circ},
\]
then $G=G'$ as measures. Taken together, these results delineate
the boundary between global non-identifiability and identifiable regimes
for finite Dirichlet mixtures. Section~\ref{sec:Implications} shows
that the non-identifiability conclusions also propagate directly to several related
constructions, yielding non-identifiability statements for unrestricted
finite generalized Dirichlet mixtures \citep{ConnorMosimann1969,wong2010gd},
Dirichlet-multinomial mixtures \citep{bouguila2008gdm,holmes2012dmm},
fixed-topic-matrix latent Dirichlet allocation kernels \citep{blei2003lda},
Beta-Liouville mixtures \citep{fan2013bl}, and inverted Beta-Liouville
mixtures \citep{hu2019ibl}.

The proofs use two complementary devices. On the fixed-total slice
$\Theta_{A}$, an additive log-ratio transform converts the Dirichlet
kernels into exponential functions, so identifiability reduces to
a classical linear-independence argument for distinct exponentials,
in the spirit of transform methods used in mixture identifiability;
see \citet[Sec.~3.1]{TitteringtonSmithMakov1985}. For box-type restrictions
and for the full-space result on $\mathcal{F}_{J-1}$, we instead
transport the kernels to the positive orthant and exploit the binomial
and multinomial theorems in local series expansions; see, e.g., \citet[Sec.~5.1]{GrahamKnuthPatashnik1994}.
The resulting local identities are then extended from a neighborhood
to the full connected domain by analytic continuation along line restrictions
(cf. \citealp[Cor.~1.2.6 and Prop.~2.2.8]{KrantzParks2002}). In Section~\ref{sec:Identifiability-of-FJ-1},
this is combined with coefficient extraction, congruence decomposition
modulo $\mathbb{Z}^{J}$, and an elementary sign counting argument
for polynomial identities on the simplex. The degree-elevation step
in that sign argument is closely related to the Bernstein--Bézier
representation and degree-raising formula on triangles; see, e.g.,
\citet[Thms.~2.4 and~2.39]{LaiSchumaker2007}. We do not know of sources
containing these ingredients in the precise forms used here. In Section~\ref{sec:Implications},
we indicate which of these mechanisms is responsible for each application.

The manuscript proceeds as follows. Section~\ref{sec:Non-identifiability}
establishes the shift identity and the resulting non-identifiability
of unrestricted finite Dirichlet mixtures. Section~\ref{sec:Identifiable-restrictions}
proves identifiability on the fixed-total class $\Theta_{A}$ and
on box-type restrictions such as $\Theta_{\mathrm{FD}}$, and records
the positive-orthant transport that also transfers these conclusions
to inverted Dirichlet mixtures. Section~\ref{sec:Identifiability-of-FJ-1}
shows that every nontrivial linear relation among Dirichlet kernels
must involve at least $J$ coefficients sharing a common sign, which
yields identifiability of mixtures with fewer than $J$ atoms on the
full parameter space. Section~\ref{sec:Implications} collects direct
consequences for related mixture classes.

\section{Non-identifiability of $\mathcal{F}$}

\label{sec:Non-identifiability}

In this section we retain the notation introduced in Section~\ref{sec:Introduction}.
For $a>0$, we write 
\[
\Gamma(a)=\int_{0}^{\infty}t^{a-1}\mathrm{e}^{-t}\,dt
\]
for the gamma function. The basic mechanism behind the non-identifiability
of $\mathcal{F}$ is the following shift identity. 
\begin{thm}
\label{thm:shift} For every $\alpha\in(0,\infty)^{J}$ and every
$x\in\Delta_{J-1}^{\circ}$, 
\[
f_{\alpha}(x)=\sum_{j=1}^{J}\frac{\alpha_{j}}{\alpha_{+}}f_{\alpha+e_{j}}(x).
\]
\end{thm}

\begin{proof}
Write $c(\alpha)=\Gamma(\alpha_{+})/\prod_{m=1}^{J}\Gamma(\alpha_{m})$
so that $f_{\alpha}(x)=c(\alpha)\prod_{m=1}^{J}x_{m}^{\alpha_{m}-1}$.

Fix $j$. Then 
\[
f_{\alpha+e_{j}}(x)=c(\alpha+e_{j})\,x_{j}^{\alpha_{j}}\prod_{m\ne j}x_{m}^{\alpha_{m}-1}=c(\alpha+e_{j})\,x_{j}\prod_{m=1}^{J}x_{m}^{\alpha_{m}-1}.
\]
Next compute $c(\alpha+e_{j})$ using $\Gamma(t+1)=t\,\Gamma(t)$
to obtain 
\[
c(\alpha+e_{j})=\frac{\Gamma(\alpha_{+}+1)}{\Gamma(\alpha_{j}+1)\prod_{m\ne j}\Gamma(\alpha_{m})}=\frac{\alpha_{+}\Gamma(\alpha_{+})}{\alpha_{j}\Gamma(\alpha_{j})\prod_{m\ne j}\Gamma(\alpha_{m})}=\frac{\alpha_{+}}{\alpha_{j}}\,c(\alpha).
\]
Therefore 
\[
\frac{\alpha_{j}}{\alpha_{+}}f_{\alpha+e_{j}}(x)=c(\alpha)\,x_{j}\prod_{m=1}^{J}x_{m}^{\alpha_{m}-1}.
\]
Summing over $j$ and using the fact that $\sum_{j=1}^{J}x_{j}=1$
on the simplex yields 
\[
\sum_{j=1}^{J}\frac{\alpha_{j}}{\alpha_{+}}f_{\alpha+e_{j}}(x)=c(\alpha)\left(\sum_{j=1}^{J}x_{j}\right)\prod_{m=1}^{J}x_{m}^{\alpha_{m}-1}=f_{\alpha}(x).
\]
\end{proof}
For $J=2$, Theorem~\ref{thm:shift} is proved in \citet{Simone2022}.
The following corollary is then immediate. 
\begin{cor}
\label{cor:nonid} The class of finite Dirichlet mixtures $\mathcal{F}$
is not identifiable. In particular, for any $\alpha\in\mathbb{R}_{>0}^{J}$
the one-atom measure $G_{0}=\delta_{\alpha}$ and the $J$-atom measure
\[
G_{1}=\sum_{j=1}^{J}\frac{\alpha_{j}}{\alpha_{+}}\,\delta_{\alpha+e_{j}}
\]
are distinct but $m_{G_{0}}=m_{G_{1}}$. 
\end{cor}

\begin{proof}
By Theorem~\ref{thm:shift}, $m_{G_{0}}(x)=f_{\alpha}(x)=m_{G_{1}}(x)$
for all $x\in\Delta_{J-1}^{\circ}$, yet $G_{0}\neq G_{1}$. 
\end{proof}

\section{Identifiable restrictions of $\mathcal{F}$}

\label{sec:Identifiable-restrictions}

We begin with a general sufficient condition for identifiability.
This criterion will also be used in later sections and is a multivariate
analogue of Theorem~3.1.1 of \citet{TitteringtonSmithMakov1985}. 
\begin{lem}
\label{lem:linind_ident} Let $\Theta\subseteq\mathbb{R}_{>0}^{J}$.
Suppose the family $\{f_{\alpha}:\alpha\in\Theta\}$ is linearly independent
in the sense that for any distinct $\alpha^{(1)},\dots,\alpha^{(K)}\in\Theta$,
$K\in\mathbb{N}$, 
\[
\sum_{k=1}^{K}c_{k}f_{\alpha^{(k)}}(x)=0\ \text{a.e. on }\Delta_{J-1}^{\circ}\quad\Longrightarrow\quad c_{1}=\cdots=c_{K}=0.
\]
Then $\mathcal{F}\left(\Theta\right)$ is identifiable. 
\end{lem}

\begin{proof}
Let 
\[
G=\sum_{k=1}^{K}\pi_{k}\delta_{\alpha^{(k)}}\quad\text{and}\quad G'=\sum_{\ell=1}^{L}\rho_{\ell}\delta_{\beta^{(\ell)}}
\]
be two finite mixing measures supported on $\Theta$, where $\pi_{k}>0$,
$\rho_{\ell}>0$, $\sum_{k=1}^{K}\pi_{k}=1$, $\sum_{\ell=1}^{L}\rho_{\ell}=1$,
and the atoms $\alpha^{(1)},\dots,\alpha^{(K)}$ are pairwise distinct,
as are $\beta^{(1)},\dots,\beta^{(L)}$.

Assume that the induced mixture densities coincide almost everywhere:
\[
m_{G}(x)=m_{G'}(x)\qquad\text{for almost every }x\in\Delta_{J-1}^{\circ}.
\]
By definition of $m_{G}$ and $m_{G'}$, this means 
\[
\sum_{k=1}^{K}\pi_{k}f_{\alpha^{(k)}}(x)=\sum_{\ell=1}^{L}\rho_{\ell}f_{\beta^{(\ell)}}(x)\qquad\text{for almost every }x\in\Delta_{J-1}^{\circ}.
\]
Move all terms to one side: 
\[
\sum_{k=1}^{K}\pi_{k}f_{\alpha^{(k)}}(x)-\sum_{\ell=1}^{L}\rho_{\ell}f_{\beta^{(\ell)}}(x)=0\qquad\text{for almost every }x\in\Delta_{J-1}^{\circ}.
\]

Now form the finite set of distinct parameters appearing in either
mixture, 
\[
\mathbb{G}=\{\alpha^{(1)},\dots,\alpha^{(K)}\}\cup\{\beta^{(1)},\dots,\beta^{(L)}\}\subseteq\Theta.
\]
Enumerate $\mathbb{G}$ as $\mathbb{G}=\{\gamma^{(1)},\dots,\gamma^{(M)}\}$
with all $\gamma^{(i)}$ distinct. For each $i\in\{1,\dots,M\}$ define
the coefficient 
\[
c_{i}=G(\{\gamma^{(i)}\})-G'(\{\gamma^{(i)}\}),
\]
that is, 
\[
c_{i}=\begin{cases}
\pi_{k} & \text{if }\gamma^{(i)}=\alpha^{(k)}\text{ and }\gamma^{(i)}\notin\{\beta^{(1)},\dots,\beta^{(L)}\},\\
-\rho_{\ell} & \text{if }\gamma^{(i)}=\beta^{(\ell)}\text{ and }\gamma^{(i)}\notin\{\alpha^{(1)},\dots,\alpha^{(K)}\},\\
\pi_{k}-\rho_{\ell} & \text{if }\gamma^{(i)}=\alpha^{(k)}=\beta^{(\ell)}.
\end{cases}
\]
With this notation, the previous identity becomes a linear combination
over distinct kernels: 
\[
\sum_{i=1}^{M}c_{i}f_{\gamma^{(i)}}(x)=0\qquad\text{for almost every }x\in\Delta_{J-1}^{\circ}.
\]

By the assumed linear independence of the family $\{f_{\alpha}:\alpha\in\Theta\}$,
applied to the distinct parameters $\gamma^{(1)},\dots,\gamma^{(M)}\in\Theta$,
we must have 
\[
c_{1}=c_{2}=\cdots=c_{M}=0.
\]
Thus $G(\{\gamma^{(i)}\})=G'(\{\gamma^{(i)}\})$ for every $i=1,\dots,M$.
Since both $G$ and $G'$ are finite discrete measures supported on
$\mathbb{G}$, agreement on all singleton masses implies $G=G'$ as
measures. 
\end{proof}

\subsection{The class $\mathcal{F}\left(\Theta_{A}\right)$}

We first provide a linear independence result for finite sums of multivariate
exponential functions. 
\begin{lem}
\label{lem:multivar_exp} Let $u_{1},\dots,u_{K}\in\mathbb{R}^{d}$
be distinct vectors and $c_{1},\dots,c_{K}\in\mathbb{R}$, $K\in\mathbb{N}$.
If 
\[
\sum_{k=1}^{K}c_{k}\mathrm{e}^{\left\langle u_{k},t\right\rangle }=0\qquad\text{for all }t\in\mathbb{R}^{d},
\]
then $c_{1}=\cdots=c_{K}=0$. 
\end{lem}

\begin{proof}
Choose $v\in\mathbb{R}^{d}$ such that the scalars $\lambda_{k}=\left\langle u_{k},v\right\rangle $
are all distinct. Such a $v$ exists because for each $i\ne j$ the
set $\{v:\left\langle u_{i}-u_{j},v\right\rangle =0\}$ is a hyperplane,
and a finite union of hyperplanes cannot cover $\mathbb{R}^{d}$.
Restricting to the line $t=sv$ gives 
\[
\sum_{k=1}^{K}c_{k}\mathrm{e}^{\lambda_{k}s}=0\qquad\forall s\in\mathbb{R}.
\]
Order the $\lambda_{k}$ so that $\lambda_{1}<\cdots<\lambda_{K}$
and divide by $\mathrm{e}^{\lambda_{K}s}$: 
\[
c_{K}+\sum_{k=1}^{K-1}c_{k}\mathrm{e}^{(\lambda_{k}-\lambda_{K})s}=0.
\]
Letting $s\to\infty$ yields $c_{K}=0$ since each $\lambda_{k}-\lambda_{K}<0$
for $k<K$. Iterate to conclude all $c_{k}=0$. 
\end{proof}
Define the additive log-ratio map $\psi:\mathbb{R}^{J-1}\to\Delta_{J-1}^{\circ}$,
where $\psi\left(t\right)=\left(x_{1}\left(t\right),\dots,x_{J}\left(t\right)\right)$,
\[
S(t)=1+\sum_{r=1}^{J-1}\mathrm{e}^{t_{r}},\qquad x_{j}(t)=\frac{\mathrm{e}^{t_{j}}}{S(t)}\ (1\le j\le J-1),\qquad x_{J}(t)=\frac{1}{S(t)}.
\]
This map is a bijection with inverse 
\[
\psi^{-1}(x)=\left(\log\frac{x_{1}}{x_{J}},\dots,\log\frac{x_{J-1}}{x_{J}}\right),\qquad x\in\Delta_{J-1}^{\circ}.
\]
We also write $\mathbf{1}\{\cdot\}$ for the indicator function. 
\begin{lem}
\label{lem:alr_jac} The Jacobian determinant of $t\mapsto(x_{1}(t),\dots,x_{J-1}(t))$
satisfies 
\[
\left|\det\left(\frac{\partial(x_{1},\dots,x_{J-1})}{\partial(t_{1},\dots,t_{J-1})}\right)\right|=\prod_{j=1}^{J}x_{j}(t).
\]
\end{lem}

\begin{proof}
For $1\le i,j\le J-1$ one computes 
\[
\frac{\partial x_{i}}{\partial t_{j}}=x_{i}(\mathbf{1}\{i=j\}-x_{j}).
\]
Let $D=\mathrm{diag}(x_{1},\dots,x_{J-1})$ and $u=(x_{1},\dots,x_{J-1})^{\top}$.
Then the Jacobian matrix is $M=D-uu^{\top}$. By the matrix determinant
lemma, $\det(D-uu^{\top})=\det(D)\left(1-u^{\top}D^{-1}u\right)$.
Since $u^{\top}D^{-1}u=\sum_{i=1}^{J-1}x_{i}$ we obtain 
\[
\det(M)=\left(\prod_{i=1}^{J-1}x_{i}\right)\left(1-\sum_{i=1}^{J-1}x_{i}\right)=\left(\prod_{i=1}^{J-1}x_{i}\right)x_{J}=\prod_{j=1}^{J}x_{j}.
\]
\end{proof}
We write $X\sim\mathrm{Dir}(\alpha)$ when $X$ has density~\eqref{eq:Dirichlet_dens},
and for $\alpha=(\alpha_{1},\dots,\alpha_{J})$ we abbreviate $\alpha_{1:J-1}=(\alpha_{1},\dots,\alpha_{J-1})$.
The next theorem shows that fixing the total mass parameter removes
the non-identifiability exhibited in Theorem~\ref{thm:shift}. 
\begin{thm}
\label{thm:fixedA} For every $A>0$, the finite Dirichlet mixture
class $\mathcal{F}\left(\Theta_{A}\right)$ is identifiable. 
\end{thm}

\begin{proof}
Let $\alpha\in\Theta_{A}$. Consider the density of $T=\psi^{-1}(X)$
when $X\sim\mathrm{Dir}(\alpha)$. By change of variables and Lemma~\ref{lem:alr_jac},
\begin{align*}
g_{\alpha}(t) & =f_{\alpha}(x(t))\,\left|\det\left(\frac{\partial(x_{1},\dots,x_{J-1})}{\partial(t_{1},\dots,t_{J-1})}\right)\right|\\
 & =\frac{\Gamma(A)}{\prod_{j=1}^{J}\Gamma(\alpha_{j})}\prod_{j=1}^{J}x_{j}(t)^{\alpha_{j}-1}\prod_{j=1}^{J}x_{j}(t)\\
 & =\frac{\Gamma(A)}{\prod_{j=1}^{J}\Gamma(\alpha_{j})}\prod_{j=1}^{J}x_{j}(t)^{\alpha_{j}}.
\end{align*}
Using $x_{j}(t)=\mathrm{e}^{t_{j}}/S(t)$ for $j\le J-1$ and $x_{J}(t)=1/S(t)$
gives 
\[
\prod_{j=1}^{J}x_{j}(t)^{\alpha_{j}}=\exp\left(\sum_{j=1}^{J-1}\alpha_{j}t_{j}\right)S(t)^{-A}=\mathrm{e}^{\left\langle \alpha_{1:J-1},t\right\rangle }S(t)^{-A},
\]
so 
\begin{equation}
g_{\alpha}(t)=\frac{\Gamma(A)}{\prod_{j=1}^{J}\Gamma(\alpha_{j})}\frac{\mathrm{e}^{\left\langle \alpha_{1:J-1},t\right\rangle }}{S(t)^{A}}.\label{eq:alr_density}
\end{equation}

Now suppose two finite mixtures supported on $\Theta_{A}$ coincide:
\[
\sum_{k=1}^{K}\pi_{k}f_{\alpha^{(k)}}(x)=\sum_{\ell=1}^{L}\rho_{\ell}f_{\beta^{(\ell)}}(x)\qquad\text{a.e. on }\Delta_{J-1}^{\circ}.
\]
Let 
\[
H(x)=\sum_{k=1}^{K}\pi_{k}f_{\alpha^{(k)}}(x)-\sum_{\ell=1}^{L}\rho_{\ell}f_{\beta^{(\ell)}}(x),\qquad x\in\Delta_{J-1}^{\circ}.
\]
Each $f_{\gamma}$ is continuous on $\Delta_{J-1}^{\circ}$, hence
$H$ is continuous on $\Delta_{J-1}^{\circ}$. Since $H(x)=0$ for
almost every $x\in\Delta_{J-1}^{\circ}$, we have $H(x)=0$ for all
$x\in\Delta_{J-1}^{\circ}$. Indeed, if $H(x_{0})\neq0$ for some
$x_{0}\in\Delta_{J-1}^{\circ}$, then continuity yields a neighbourhood
$V\subset\Delta_{J-1}^{\circ}$ of $x_{0}$ such that $|H(x)|\ge|H(x_{0})|/2$
for all $x\in V$. This set $V$ has positive Lebesgue measure, contradicting
$H=0$ a.e. Transforming by $x=\psi(t)$, using \eqref{eq:alr_density},
and canceling the common factor $\Gamma(A)$ yields for all $t\in\mathbb{R}^{J-1}$,
\[
\sum_{k=1}^{K}\pi_{k}\frac{1}{\prod_{j=1}^{J}\Gamma(\alpha_{j}^{(k)})}\mathrm{e}^{\left\langle \alpha_{1:J-1}^{(k)},t\right\rangle }=\sum_{\ell=1}^{L}\rho_{\ell}\frac{1}{\prod_{j=1}^{J}\Gamma(\beta_{j}^{(\ell)})}\mathrm{e}^{\left\langle \beta_{1:J-1}^{(\ell)},t\right\rangle }.
\]
Bring all terms to one side and merge equal exponent vectors. We obtain
a finite identity of the form 
\[
\sum_{i=1}^{N}c_{i}\mathrm{e}^{\left\langle u_{i},t\right\rangle }=0\qquad\forall t\in\mathbb{R}^{J-1},
\]
for $N\le K+L$, with distinct $u_{i}\in\mathbb{R}^{J-1}$. By Lemma~\ref{lem:multivar_exp},
all $c_{i}=0$. In particular, an exponent vector can appear on one
side only if it also appears on the other side with the same coefficient.
Because the parameters in each mixture are distinct and the last coordinate
is determined by the fixed-total constraint $\alpha_{+}=A$, the exponent
vectors $\alpha_{1:J-1}^{(k)}$ are pairwise distinct, and likewise
the vectors $\beta_{1:J-1}^{(\ell)}$ are pairwise distinct. Therefore
the two supports in $\mathbb{R}^{J-1}$ coincide: 
\[
\{\alpha_{1:J-1}^{(k)}\}_{k=1}^{K}=\{\beta_{1:J-1}^{(\ell)}\}_{\ell=1}^{L},
\]
so in particular $K=L$ and after relabeling we may assume $\alpha_{1:J-1}^{(k)}=\beta_{1:J-1}^{(k)}$,
for all $k$. Since $\alpha_{+}^{(k)}=\beta_{+}^{(k)}=A$, the last
coordinates are forced: 
\[
\alpha_{J}^{(k)}=A-\sum_{j=1}^{J-1}\alpha_{j}^{(k)}=A-\sum_{j=1}^{J-1}\beta_{j}^{(k)}=\beta_{J}^{(k)},
\]
so $\alpha^{(k)}=\beta^{(k)}$. Finally, the corresponding coefficients
must match, giving $\pi_{k}=\rho_{k}$. Thus the mixing measures are
equal. 
\end{proof}
The following corollary will prove useful in the sequel. 
\begin{cor}
\label{cor:fixedA_linind} For every $A>0$, the family $\{f_{\alpha}:\alpha\in\Theta_{A}\}$
is linearly independent. 
\end{cor}

\begin{proof}
Let $\alpha^{(1)},\dots,\alpha^{(M)}\in\Theta_{A}$ be distinct and
suppose that 
\[
\sum_{i=1}^{M}c_{i}f_{\alpha^{(i)}}(x)=0\qquad\text{a.e. on }\Delta_{J-1}^{\circ}.
\]
Define 
\[
I_{+}=\{i:c_{i}>0\},\qquad I_{-}=\{i:c_{i}<0\}.
\]
Integrating both sides over $\Delta_{J-1}^{\circ}$ and using that
each $f_{\alpha^{(i)}}$ is a density gives 
\[
\sum_{i\in I_{+}}c_{i}=-\sum_{i\in I_{-}}c_{i}=s\ge0.
\]
If $s=0$, then both sums are $0$, so $I_{+}=I_{-}=\varnothing$
and hence $c_{1}=\cdots=c_{M}=0$. Assume therefore that $s>0$. Define
two finite mixing measures on $\Theta_{A}$ by 
\[
G_{+}=\sum_{i\in I_{+}}\frac{c_{i}}{s}\,\delta_{\alpha^{(i)}}\qquad\text{and}\qquad G_{-}=\sum_{i\in I_{-}}\frac{-c_{i}}{s}\,\delta_{\alpha^{(i)}}.
\]
Then $m_{G_{+}}=m_{G_{-}}$ almost everywhere on $\Delta_{J-1}^{\circ}$,
so Theorem~\ref{thm:fixedA} implies $G_{+}=G_{-}$. Because the
parameters $\alpha^{(1)},\dots,\alpha^{(M)}$ are distinct and $I_{+}\cap I_{-}=\varnothing$,
this equality of measures forces every coefficient $c_{i}$ to be
$0$. Thus the family is linearly independent. 
\end{proof}

\subsection{The class $\mathcal{F}\left(\Theta_{\mathrm{FD}}\right)$}

Write $d=J-1$. In this subsection we use standard multi-index notation
on $\mathbb{R}^{d}$. Thus $\mathbb{N}_{0}=\{0\}\cup\mathbb{N}$,
and for $m=(m_{1},\dots,m_{d})\in\mathbb{N}_{0}^{d}$ we write $|m|=\sum_{j=1}^{d}m_{j}$,
$m!=\prod_{j=1}^{d}m_{j}!$, and $y^{m}=\prod_{j=1}^{d}y_{j}^{m_{j}}$.
The proof uses the following two elementary lemmas. 
\begin{lem}
\label{lem:disjoint} Let $u,u'\in(-1,0)^{d}$ and $m,m'\in\mathbb{N}_{0}^{d}$.
If $u+m=u'+m'$, then $u=u'$ and $m=m'$. 
\end{lem}

\begin{proof}
Rearrange to $u-u'=m'-m\in\mathbb{Z}^{d}$. But $u-u'\in(-1,1)^{d}$,
and the only integer vector in $(-1,1)^{d}$ is $0$. Hence $u=u'$
and then $m=m'$. 
\end{proof}
\begin{lem}
\label{lem:dirichlet_series} Let $\{\lambda_{n}\}_{n\ge1}$ be a
strictly increasing sequence of real numbers with $\lambda_{n}\to\infty$.
Suppose $\sum_{n\ge1}a_{n}\mathrm{e}^{-\lambda_{n}s}$ converges absolutely
for all $s\ge s_{0}$ and 
\[
\sum_{n\ge1}a_{n}\mathrm{e}^{-\lambda_{n}s}=0\qquad\forall s\ge s_{0}.
\]
Then $a_{n}=0$ for all $n$. 
\end{lem}

\begin{proof}
Multiply by $\mathrm{e}^{\lambda_{1}s}$ to get 
\[
a_{1}+\sum_{n\ge2}a_{n}\mathrm{e}^{-(\lambda_{n}-\lambda_{1})s}=0.
\]
As $s\to\infty$, $\mathrm{e}^{-(\lambda_{n}-\lambda_{1})s}\to0$
for all $n\ge2$. Moreover, for all $s\ge s_{0}$ and $n\ge2$, 
\[
\left|a_{n}\mathrm{e}^{-(\lambda_{n}-\lambda_{1})s}\right|\le|a_{n}|\mathrm{e}^{-(\lambda_{n}-\lambda_{1})s_{0}},
\]
and the series $\sum_{n\ge2}|a_{n}|\mathrm{e}^{-(\lambda_{n}-\lambda_{1})s_{0}}$
converges by the assumed absolute convergence at $s_{0}$. Dominated
convergence therefore justifies passing the limit inside the sum.
Hence $a_{1}=0$. Subtract the $n=1$ term and repeat, inductively. 
\end{proof}
Define the chart $\varphi:\mathbb{R}_{>0}^{d}\to\Delta_{J-1}^{\circ}$
by $\varphi\left(y\right)=\left(x_{1}\left(y\right),\dots,x_{J}\left(y\right)\right)$,
\[
Y=\sum_{j=1}^{d}y_{j},\qquad x_{j}(y)=\frac{y_{j}}{1+Y}\ (1\le j\le d),\qquad x_{J}(y)=\frac{1}{1+Y}.
\]
This map is a bijection from $\mathbb{R}_{>0}^{d}$ onto $\Delta_{J-1}^{\circ}$. 
\begin{lem}
\label{lem:chart_jac} The Jacobian determinant of $y\mapsto(x_{1}(y),\dots,x_{d}(y))$
is $(1+Y)^{-J}$. 
\end{lem}

\begin{proof}
For $1\le i,j\le d$, 
\[
\frac{\partial x_{i}}{\partial y_{j}}=\frac{\delta_{ij}(1+Y)-y_{i}}{(1+Y)^{2}}.
\]
Thus the Jacobian matrix equals $(1+Y)^{-2}\left((1+Y)I-y\mathbf{1}^{\top}\right)$.
Factor $(1+Y)$ to get 
\[
\det\left(\frac{\partial x}{\partial y}\right)=(1+Y)^{-2d}(1+Y)^{d}\det\left(I-(1+Y)^{-1}y\mathbf{1}^{\top}\right).
\]
By the matrix determinant lemma, $\det(I-ab^{\top})=1-b^{\top}a$.
Here $a=(1+Y)^{-1}y$ and $b=\mathbf{1}$, so $b^{\top}a=(1+Y)^{-1}\sum_{j=1}^{d}y_{j}=Y/(1+Y)$.
Hence the determinant is $1-Y/(1+Y)=1/(1+Y)$. Combining factors gives
$(1+Y)^{-2d}(1+Y)^{d}(1+Y)^{-1}=(1+Y)^{-(d+1)}=(1+Y)^{-J}$. 
\end{proof}
For $\alpha\in(0,\infty)^{J}$, define the transported kernel on $\mathbb{R}_{>0}^{d}$
by 
\[
h_{\alpha}(y)=f_{\alpha}(x(y))\left|\det\left(\frac{\partial(x_{1},\dots,x_{d})}{\partial(y_{1},\dots,y_{d})}\right)\right|.
\]
By Lemma~\ref{lem:chart_jac}, for $y\in\mathbb{R}_{>0}^{d}$, 
\begin{equation}
h_{\alpha}(y)=\frac{\Gamma(\alpha_{+})}{\prod_{j=1}^{J}\Gamma(\alpha_{j})}\frac{\prod_{j=1}^{d}y_{j}^{\alpha_{j}-1}}{(1+Y)^{\alpha_{+}}}.\label{eq:transported_kernel}
\end{equation}
Equation~\eqref{eq:transported_kernel} is the density of the inverted
Dirichlet distribution; see, for example, \citet[Ch. 49]{kotz2019continuous}.
When $J=2$, this reduces to the beta prime distribution. 
\begin{thm}
\label{thm:linind_FD} The family $\{f_{\alpha}:\alpha\in\Theta_{\mathrm{FD}}\}$
is linearly independent. Consequently, the finite Dirichlet mixture
class $\mathcal{F}\left(\Theta_{\mathrm{FD}}\right)$ is identifiable. 
\end{thm}

\begin{proof}
By Lemma~\ref{lem:linind_ident}, it is enough to prove linear independence.
Let $\alpha^{(1)},\dots,\alpha^{(K)}\in\Theta_{\mathrm{FD}}$ be distinct
and suppose that 
\[
\sum_{k=1}^{K}c_{k}f_{\alpha^{(k)}}(x)=0\qquad\text{for a.e. }x\in\Delta_{J-1}^{\circ}.
\]
The left-hand side is continuous on $\Delta_{J-1}^{\circ}$, so the
identity in fact holds for all $x\in\Delta_{J-1}^{\circ}$.

Transport this identity by the bijection $y\mapsto\varphi(y)$. Multiplying
by the Jacobian determinant from Lemma~\ref{lem:chart_jac}, which
is strictly positive and does not depend on $\alpha$, yields 
\begin{equation}
\sum_{k=1}^{K}c_{k}h_{\alpha^{(k)}}(y)=0\qquad\text{for all }y\in\mathbb{R}_{>0}^{d}.\label{eq:FD_transport_identity}
\end{equation}
Fix the open neighbourhood 
\[
U=\{y\in\mathbb{R}_{>0}^{d}:Y<1/2\}.
\]
For $v>0$ and $y\in U$, the binomial series gives 
\begin{equation}
(1+Y)^{-v}=\sum_{n=0}^{\infty}(-1)^{n}\frac{(v)_{n}}{n!}Y^{n},\qquad(v)_{n}=v(v+1)\cdots(v+n-1),\label{eq:FD_binomial}
\end{equation}
and for each $n\in\mathbb{N}_{0}$ the multinomial theorem gives 
\begin{equation}
Y^{n}=(y_{1}+\cdots+y_{d})^{n}=\sum_{\substack{m\in\mathbb{N}_{0}^{d}\\
|m|=n
}
}\frac{n!}{m!}y^{m}.\label{eq:FD_multinomial}
\end{equation}
Now fix $\alpha\in\Theta_{\mathrm{FD}}$ and set 
\[
u(\alpha)=(\alpha_{1}-1,\dots,\alpha_{d}-1)\in(-1,0)^{d},\qquad v(\alpha)=\alpha_{+}.
\]
Substituting \eqref{eq:FD_binomial} and \eqref{eq:FD_multinomial}
into \eqref{eq:transported_kernel} yields, for $y\in U$, 
\begin{equation}
h_{\alpha}(y)=\frac{\Gamma(v(\alpha))}{\prod_{j=1}^{J}\Gamma(\alpha_{j})}\sum_{m\in\mathbb{N}_{0}^{d}}(-1)^{|m|}\frac{(v(\alpha))_{|m|}}{m!}y^{u(\alpha)+m}.\label{eq:FD_series_expansion}
\end{equation}
For each fixed $y\in U$, the series in \eqref{eq:FD_series_expansion}
is absolutely convergent. Indeed, using \eqref{eq:FD_multinomial},
\[
\sum_{m\in\mathbb{N}_{0}^{d}}\frac{(v(\alpha))_{|m|}}{m!}y^{m}=\sum_{n=0}^{\infty}(v(\alpha))_{n}\sum_{\substack{m\in\mathbb{N}_{0}^{d}\\
|m|=n
}
}\frac{y^{m}}{m!}=\sum_{n=0}^{\infty}(v(\alpha))_{n}\frac{Y^{n}}{n!},
\]
and the last series converges absolutely because $Y<1$. Since $y^{u(\alpha)}=\prod_{j=1}^{d}y_{j}^{\alpha_{j}-1}$
is finite for every $y\in U$, \eqref{eq:FD_series_expansion} is
pointwise absolutely convergent on $U$.

Substituting \eqref{eq:FD_series_expansion} into \eqref{eq:FD_transport_identity}
is legitimate for each $y\in U$, since we are summing only finitely
many pointwise absolutely convergent series. Moreover, if
\[
u(\alpha^{(k)})+m=u(\alpha^{(k')})+m',
\]
then
\[
u(\alpha^{(k)})-u(\alpha^{(k')})=m'-m\in\mathbb{Z}^{d},
\]
so Lemma~\ref{lem:disjoint} implies that $u(\alpha^{(k)})=u(\alpha^{(k')})$,
and then necessarily $m=m'$. Thus coincidences among exponent vectors
can occur only among terms having the same value of $u(\alpha)$, and
regrouping first by the distinct values of $u(\alpha)$ is unambiguous.
Let $u_{1},\dots,u_{L}\in(-1,0)^{d}$ be the distinct values taken
by $u(\alpha^{(k)})$, and let $\mathcal{G}_{g}$ be the set of indices
$k$ such that $u(\alpha^{(k)})=u_{g}$. For $k\in\mathcal{G}_{g}$,
write 
\[
v_{k}=\alpha_{+}^{(k)},\qquad d_{k}=c_{k}\frac{\Gamma(v_{k})}{\prod_{j=1}^{J}\Gamma(\alpha_{j}^{(k)})}.
\]
Then we obtain 
\begin{equation}
\sum_{g=1}^{L}\sum_{m\in\mathbb{N}_{0}^{d}}A_{g,m}y^{u_{g}+m}=0\qquad\text{for all }y\in U,\label{eq:FD_regrouped}
\end{equation}
where 
\begin{equation}
A_{g,m}=\sum_{k\in\mathcal{G}_{g}}d_{k}(-1)^{|m|}\frac{(v_{k})_{|m|}}{m!}.\label{eq:FD_Agm}
\end{equation}
By Lemma~\ref{lem:disjoint}, all exponent vectors $u_{g}+m$ appearing
in \eqref{eq:FD_regrouped} are distinct.

Let 
\[
W=\{u_{g}+m:1\le g\le L,\ m\in\mathbb{N}_{0}^{d}\}.
\]
This is a countable subset of $\mathbb{R}^{d}$. For any two distinct
vectors $w,w'\in W$, the constraint $\langle w,\lambda\rangle=\langle w',\lambda\rangle$
defines a hyperplane in $\mathbb{R}^{d}$. The union of these hyperplanes
over all pairs $w\neq w'$ is therefore a countable union of Lebesgue-null
sets, hence cannot cover the positive orthant $\mathbb{R}_{>0}^{d}$.
Choose $\lambda\in\mathbb{R}_{>0}^{d}$ outside this union. Then the
scalar products $\langle w,\lambda\rangle$ are pairwise distinct
for $w\in W$.

For $s\ge0$, define 
\[
y(s)=(\mathrm{e}^{-\lambda_{1}s},\dots,\mathrm{e}^{-\lambda_{d}s}).
\]
Since each $\lambda_{j}>0$, we have $Y(s)=\sum_{j=1}^{d}\mathrm{e}^{-\lambda_{j}s}\to0$,
so $y(s)\in U$ for all $s\ge s_{0}$ for some $s_{0}$. Substituting
$y=y(s)$ into \eqref{eq:FD_regrouped} gives 
\begin{equation}
\sum_{g=1}^{L}\sum_{m\in\mathbb{N}_{0}^{d}}A_{g,m}\mathrm{e}^{-\mu_{g,m}s}=0\qquad\text{for all }s\ge s_{0},\label{eq:FD_dirichlet_series}
\end{equation}
where $\mu_{g,m}=\langle u_{g}+m,\lambda\rangle$. For each fixed
$s\ge s_{0}$, the series in \eqref{eq:FD_dirichlet_series} is absolutely
convergent because \eqref{eq:FD_regrouped} is absolutely convergent
at $y(s)\in U$. Moreover, the exponents $\mu_{g,m}$ are all distinct,
and they can be arranged into a strictly increasing sequence diverging
to $\infty$. Indeed, if $\lambda_{\min}=\min_{1\le j\le d}\lambda_{j}$
and $C=\min_{1\le g\le L}\langle u_{g},\lambda\rangle$, then 
\[
\mu_{g,m}\ge C+\langle m,\lambda\rangle\ge C+\lambda_{\min}|m|,
\]
so only finitely many pairs $(g,m)$ satisfy $\mu_{g,m}\le M$ for
a given $M$. Lemma~\ref{lem:dirichlet_series} therefore applies
to \eqref{eq:FD_dirichlet_series} and yields 
\begin{equation}
A_{g,m}=0\qquad\text{for all }g\in\{1,\dots,L\}\text{ and all }m\in\mathbb{N}_{0}^{d}.\label{eq:FD_Agm_zero}
\end{equation}

Now fix $g\in\{1,\dots,L\}$. Combining \eqref{eq:FD_Agm} and \eqref{eq:FD_Agm_zero}
gives 
\begin{equation}
\sum_{k\in\mathcal{G}_{g}}d_{k}(v_{k})_{|m|}=0\qquad\text{for all }m\in\mathbb{N}_{0}^{d}.\label{eq:FD_group_pochhammer}
\end{equation}
In particular, taking $m=(n,0,\dots,0)$ gives 
\begin{equation}
\sum_{k\in\mathcal{G}_{g}}d_{k}(v_{k})_{n}=0\qquad\text{for all }n\in\mathbb{N}_{0}.\label{eq:FD_group_pochhammer_n}
\end{equation}
Multiply \eqref{eq:FD_group_pochhammer_n} by $z^{n}/n!$ and sum
over $n\ge0$. Since $\sum_{n\ge0}(v)_{n}z^{n}/n!=(1-z)^{-v}$ for
$|z|<1$, we obtain 
\begin{equation}
\sum_{k\in\mathcal{G}_{g}}d_{k}(1-z)^{-v_{k}}=0\qquad\text{for all }z\in(0,1).\label{eq:FD_group_binomial}
\end{equation}
Now set $r=-\log(1-z)\in\mathbb{R}_{>0}$. Then \eqref{eq:FD_group_binomial}
becomes 
\begin{equation}
\sum_{k\in\mathcal{G}_{g}}d_{k}\mathrm{e}^{v_{k}r}=0\qquad\text{for all }r>0.\label{eq:FD_group_exponential}
\end{equation}
The scalars $\{v_{k}:k\in\mathcal{G}_{g}\}$ are distinct. Indeed,
within a fixed group $\mathcal{G}_{g}$ the first $d=J-1$ coordinates
of $\alpha^{(k)}$ are fixed, so equality $v_{k}=v_{k'}$ would also
force 
\[
\alpha_{J}^{(k)}=v_{k}-\sum_{j=1}^{d}\alpha_{j}^{(k)}=v_{k'}-\sum_{j=1}^{d}\alpha_{j}^{(k')}=\alpha_{J}^{(k')},
\]
and hence $\alpha^{(k)}=\alpha^{(k')}$, contrary to the distinctness
of the parameters. Writing $\mathcal{G}_{g}=\{k_{1},\dots,k_{r}\}$
and ordering so that $v_{k_{1}}<\cdots<v_{k_{r}}$, we divide \eqref{eq:FD_group_exponential}
by $\mathrm{e}^{v_{k_{r}}r}$ and let $r\to\infty$. This shows that
the coefficient of $\mathrm{e}^{v_{k_{r}}r}$ is $0$. Repeating the
argument gives $d_{k}=0$ for every $k\in\mathcal{G}_{g}$.

Finally, the factor $\Gamma(v_{k})/\prod_{j=1}^{J}\Gamma(\alpha_{j}^{(k)})$
is nonzero, so $d_{k}=0$ implies $c_{k}=0$ for all $k\in\mathcal{G}_{g}$.
Since $g$ was arbitrary, we conclude that $c_{1}=\cdots=c_{K}=0$.
This proves linear independence, and the identifiability statement
then follows from Lemma~\ref{lem:linind_ident}. 
\end{proof}
\begin{rem}
\label{rem:interval_box} The proof uses only that $u(\alpha)=(\alpha_{1}-1,\dots,\alpha_{J-1}-1)$
lies in a box of side length strictly less than $1$ so that Lemma~\ref{lem:disjoint}
applies. More generally, one may replace $(0,1)$ by intervals $I_{j}$
of length $<1$ for $j=1,\dots,J-1$ and obtain identifiability on
\[
\{\alpha:\alpha_{j}\in I_{j}\ (1\le j\le J-1),\alpha_{J}>0\},
\]
and similarly after choosing any other baseline coordinate in place
of $J$. 
\end{rem}

\begin{rem}
\label{rem:inverted_transfer} Because $\varphi$ is bijective and
the Jacobian factor in the definition of $h_{\alpha}$ does not depend
on $\alpha$, a finite linear identity among the kernels $\{f_{\alpha}:\alpha\in\Theta\}$
on $\Delta_{J-1}^{\circ}$ is equivalent to the corresponding identity
among the transported kernels $\{h_{\alpha}:\alpha\in\Theta\}$ on
$\mathbb{R}_{>0}^{d}$. Consequently, every identifiability or non-identifiability
statement for finite Dirichlet mixtures supported on $\Theta\subseteq\mathbb{R}_{>0}^{J}$
transfers verbatim to the corresponding class of finite inverted Dirichlet
mixtures, and conversely. 
\end{rem}

\section{Identifiability of $\mathcal{F}_{J-1}$}

\label{sec:Identifiability-of-FJ-1}

Section~\ref{sec:Non-identifiability} shows that on the full parameter
space $\Theta=\mathbb{R}_{>0}^{J}$ the Dirichlet family is not linearly
independent: the unit-shift relation of Theorem~\ref{thm:shift}
yields an identity between one Dirichlet kernel and $J$ Dirichlet
kernels. In this section we show that, in a precise sense, this is
the minimal obstruction detected by our analysis. More exactly, every
nontrivial linear identity among Dirichlet kernels must involve at least
$J$ coefficients of the same sign. It follows that mixtures with fewer
than $J$ components are identifiable on the full parameter space.

We continue to write $d=J-1$ and to use the chart $\varphi$ and
the transported kernels $h_{\alpha}$ introduced in Section~\ref{sec:Identifiable-restrictions}.
Since Lemma~\ref{lem:chart_jac} gives 
\[
\left|\det\left(\frac{\partial(x_{1},\dots,x_{d})}{\partial(y_{1},\dots,y_{d})}\right)\right|=(1+Y)^{-J}=x_{J}(y)^{J},
\]
equation~\eqref{eq:transported_kernel} yields 
\begin{equation}
f_{\alpha}(x(y))=x_{J}(y)^{-J}h_{\alpha}(y),\qquad x_{J}(y)=\frac{1}{1+Y}.\label{eq:f_vs_h}
\end{equation}
In particular, a finite linear identity among the kernels $\{f_{\alpha}\}$
is equivalent to the corresponding identity among the transported
kernels $\{h_{\alpha}\}$, and the sign pattern of the coefficients
is unchanged.

The argument proceeds in three steps. We first extract coefficients
from the local expansion of a linear combination of transported kernels
near $y=0$. We then show that any such identity splits according
to the congruence class of the parameter modulo $\mathbb{Z}^{J}$.
Finally, we reduce the problem to a combinatorial statement about
polynomial identities on the simplex.

Let 
\[
U=\{y\in\mathbb{R}_{>0}^{d}:Y<1/2\}.
\]
For $\alpha\in\mathbb{R}_{>0}^{J}$, equation~\eqref{eq:transported_kernel}
together with the binomial series and the multinomial theorem yields
the absolutely convergent expansion 
\begin{equation}
h_{\alpha}(y)=\frac{\Gamma(\alpha_{+})}{\prod_{j=1}^{J}\Gamma(\alpha_{j})}\sum_{m\in\mathbb{N}_{0}^{d}}\frac{(-1)^{|m|}(\alpha_{+})_{|m|}}{m!}\,y^{(\alpha_{1}-1,\dots,\alpha_{d}-1)+m},\qquad y\in U.\label{eq:h_series_full}
\end{equation}
The following suite of results implements this program. 
\begin{lem}
\label{lem:coeff_extract_full} Let $\alpha^{(1)},\ldots,\alpha^{(M)}\in\mathbb{R}_{>0}^{J}$
be distinct and let $c_{1},\ldots,c_{M}\in\mathbb{R}$. Assume 
\begin{equation}
\sum_{i=1}^{M}c_{i}\,h_{\alpha^{(i)}}(y)=0,\qquad\text{for all }y\in\mathbb{R}_{>0}^{d}.\label{eq:lincomb_h_zero}
\end{equation}
For $i=1,\ldots,M$ define 
\[
u_{i}=(\alpha_{1}^{(i)}-1,\ldots,\alpha_{d}^{(i)}-1)\in(-1,\infty)^{d},\qquad v_{i}=\alpha_{+}^{(i)}>0,\qquad b_{i}=c_{i}\,\frac{\Gamma(v_{i})}{\prod_{j=1}^{J}\Gamma(\alpha_{j}^{(i)})}.
\]
Then for every $w\in(-1,\infty)^{d}$, 
\begin{equation}
\sum_{i=1}^{M}\;\sum_{\substack{m\in\mathbb{N}_{0}^{d}\\
u_{i}+m=w
}
}b_{i}\,\frac{(-1)^{|m|}\,(v_{i})_{|m|}}{m!}=0,\label{eq:coeff_extract_full}
\end{equation}
and for fixed $w$, the inner sum is finite. 
\end{lem}

\begin{proof}
By~\eqref{eq:lincomb_h_zero} the identity holds on $U$. Substituting
the absolutely convergent expansion~\eqref{eq:h_series_full} into
the finite sum in~\eqref{eq:lincomb_h_zero} and regrouping terms
by exponent vector gives 
\[
0=\sum_{i=1}^{M}b_{i}\sum_{m\in\mathbb{N}_{0}^{d}}\frac{(-1)^{|m|}(v_{i})_{|m|}}{m!}\,y^{u_{i}+m}=\sum_{w\in W}A_{w}\,y^{w},\qquad y\in U,
\]
where 
\[
W=\{u_{i}+m:\ 1\le i\le M,\ m\in\mathbb{N}_{0}^{d}\},
\]
and 
\[
A_{w}=\sum_{i=1}^{M}\;\sum_{\substack{m\in\mathbb{N}_{0}^{d}\\
u_{i}+m=w
}
}b_{i}\,\frac{(-1)^{|m|}(v_{i})_{|m|}}{m!}.
\]
For fixed $w$, there are only finitely many contributing pairs $(i,m)$,
since for each $i$ there is at most one $m$ with $u_{i}+m=w$.

Choose $\lambda\in\mathbb{R}_{>0}^{d}$ such that the scalars $\{\langle w,\lambda\rangle:w\in W\}$
are all distinct. Such a choice exists because for any two distinct
$w,w'\in W$, the constraint $\langle w,\lambda\rangle=\langle w',\lambda\rangle$
defines a hyperplane in $\mathbb{R}^{d}$, and the countable union
of these hyperplanes cannot cover $\mathbb{R}_{>0}^{d}$.

For $s\ge0$ set $y(s)=(\mathrm{e}^{-\lambda_{1}s},\ldots,\mathrm{e}^{-\lambda_{d}s})$.
Since each $\lambda_{j}>0$, we have $Y(s)=\sum_{j=1}^{d}\mathrm{e}^{-\lambda_{j}s}\to0$,
so $y(s)\in U$ for all $s\ge s_{0}$ for some $s_{0}$. Substituting
$y=y(s)$ yields 
\[
0=\sum_{w\in W}A_{w}\,y(s)^{w}=\sum_{w\in W}A_{w}\,\mathrm{e}^{-\langle w,\lambda\rangle s}\qquad\text{for all }s\ge s_{0}.
\]
This series is absolutely convergent for every $s\ge s_{0}$ because
$\sum_{w\in W}|A_{w}|\,y(s)^{w}$ is dominated by the sum of the absolute
values of the finitely many absolutely convergent series obtained from~\eqref{eq:h_series_full}
at $y=y(s)$.

It remains to verify that the exponents can be arranged into a strictly
increasing sequence diverging to $\infty$. Let 
\[
\lambda_{\min}=\min_{1\le j\le d}\lambda_{j}>0,\qquad C=\min_{1\le i\le M}\langle u_{i},\lambda\rangle.
\]
If $w=u_{i}+m$, then 
\[
\langle w,\lambda\rangle=\langle u_{i},\lambda\rangle+\langle m,\lambda\rangle\ge C+\lambda_{\min}|m|.
\]
Hence, for any fixed $M_{0}\in\mathbb{R}$, only finitely many $w\in W$
satisfy $\langle w,\lambda\rangle\le M_{0}$. Since the exponents
are also pairwise distinct by construction, they can be enumerated
as a strictly increasing sequence tending to $\infty$. Lemma~\ref{lem:dirichlet_series}
therefore applies and yields $A_{w}=0$ for every $w\in W$. Unwinding
the definition of $A_{w}$ gives~\eqref{eq:coeff_extract_full}. 
\end{proof}
The next lemma shows that a linear relation among transported kernels
splits according to the congruence class of the parameter, modulo
$\mathbb{Z}^{J}$. The coefficient-extraction lemma first separates
indices according to the first $d$ coordinates modulo integers, and
a further one-dimensional argument then separates the last coordinate.
This two-stage separation argument is reminiscent of the bookkeeping
of exponent classes in Puiseux-series methods (cf. \citealp[Sec.~4.2]{KrantzParks2002}). 
\begin{lem}
\label{lem:full_congruence} Let $\alpha^{(1)},\ldots,\alpha^{(M)}\in\mathbb{R}_{>0}^{J}$
be distinct and let $c_{1},\ldots,c_{M}\in\mathbb{R}$. Assume 
\[
\sum_{i=1}^{M}c_{i}\,h_{\alpha^{(i)}}(y)=0,\qquad y\in\mathbb{R}_{>0}^{d}.
\]
Partition $\{1,\ldots,M\}$ by the equivalence relation 
\[
i\approx i'\quad\Longleftrightarrow\quad\alpha^{(i)}-\alpha^{(i')}\in\mathbb{Z}^{J}.
\]
Then for every $\approx$-equivalence class $D$ we have 
\[
\sum_{i\in D}c_{i}\,h_{\alpha^{(i)}}(y)=0,\qquad y\in\mathbb{R}_{>0}^{d}.
\]
\end{lem}

\begin{proof}
Recall the notation from Lemma~\ref{lem:coeff_extract_full}. The
assumed identity is equivalent to 
\begin{equation}
\sum_{i=1}^{M}b_{i}\,y^{u_{i}}(1+Y)^{-v_{i}}=0,\qquad y\in\mathbb{R}_{>0}^{d}.\label{eq:full_congruence_start}
\end{equation}

First partition the index set by the weaker relation 
\[
i\sim i'\quad\Longleftrightarrow\quad u_{i}-u_{i'}\in\mathbb{Z}^{d}.
\]
Fix one $\sim$-class $C$. Choose its unique representative $u_{0}\in(-1,0]^{d}$.
Then for each $i\in C$ there is a unique $n_{i}\in\mathbb{N}_{0}^{d}$
such that 
\[
u_{i}=u_{0}+n_{i}.
\]
Consider the partial sum 
\[
F_{C}(y)=\sum_{i\in C}b_{i}\,y^{n_{i}}(1+Y)^{-v_{i}}.
\]
We claim that $F_{C}=0$ on $\mathbb{R}_{>0}^{d}$.

By Lemma~\ref{lem:coeff_extract_full}, for each exponent $w\in(-1,\infty)^{d}$
the coefficient of the generalised monomial $y^{w}$ in the expansion
is 
\[
\sum_{i=1}^{M}\;\sum_{\substack{m\in\mathbb{N}_{0}^{d}\\
u_{i}+m=w
}
}b_{i}\,\frac{(-1)^{|m|}(v_{i})_{|m|}}{m!},
\]
and Lemma~\ref{lem:coeff_extract_full} asserts that this coefficient
is equal to $0$ for every $w\in(-1,\infty)^{d}$.

Now fix $n\in\mathbb{N}_{0}^{d}$ and take $w=u_{0}+n$. Consider
a fixed index $i\in\{1,\dots,M\}$. The inner sum over $m$ contributes
only if there exists some $m\in\mathbb{N}_{0}^{d}$ such that 
\[
u_{i}+m=u_{0}+n.
\]
Rearranging gives 
\[
u_{i}-u_{0}=n-m.
\]
Since $n\in\mathbb{N}_{0}^{d}$ and $m\in\mathbb{N}_{0}^{d}$, we
have $n-m\in\mathbb{Z}^{d}$. Hence the existence of such an $m$
forces $u_{i}-u_{0}\in\mathbb{Z}^{d}$, which means that necessarily
$i\in C$. In particular, the coefficient of $y^{u_{0}+n}$ only depends
on indices $i\in C$.

For each $i\in C$, write $n_{i}=u_{i}-u_{0}\in\mathbb{Z}^{d}$, so
that $u_{i}=u_{0}+n_{i}$. Then the condition $u_{i}+m=u_{0}+n$ is
equivalent to $n_{i}+m=n$. Therefore Lemma~\ref{lem:coeff_extract_full}
yields 
\[
\sum_{i\in C}\;\sum_{\substack{m\in\mathbb{N}_{0}^{d}\\
n_{i}+m=n
}
}b_{i}\,\frac{(-1)^{|m|}(v_{i})_{|m|}}{m!}=0,\qquad n\in\mathbb{N}_{0}^{d}.
\]
But these are exactly the coefficients in the absolutely convergent
series representation of $F_{C}$ on the neighborhood $U=\{y>0:Y<1/2\}$.
Therefore every coefficient in that representation vanishes. Since the
series converges to $F_{C}(y)$ for each $y\in U$, it follows that $F_{C}(y)=0$
for every $y\in U$. Fix any $y^{*}\in U$ and any $y\in\mathbb{R}_{>0}^{d}$,
and consider 
\[
g(t)=F_{C}\bigl((1-t)y^{*}+ty\bigr),\qquad t\in[0,1].
\]
Since $\mathbb{R}_{>0}^{d}$ is convex, the whole segment lies in
$\mathbb{R}_{>0}^{d}$, and Proposition~2.2.8 of \citet{KrantzParks2002}
implies that $g$ is real-analytic on an open interval containing
$[0,1]$. Because $F_{C}=0$ on the open set $U$, we have $g(t)=0$
for all sufficiently small $t\ge0$; hence \citet[Cor.~1.2.6]{KrantzParks2002}
gives $g(t)=0$ for every $t\in[0,1]$. In particular, $F_{C}(y)=g(1)=0$.
Since $y\in\mathbb{R}_{>0}^{d}$ was arbitrary, $F_{C}=0$ on $\mathbb{R}_{>0}^{d}$.

Thus \eqref{eq:full_congruence_start} splits into a sum of identities,
one for each $\sim$-class. Fix one such class $C$ and write its
identity as 
\begin{equation}
\sum_{i\in C}b_{i}\,y^{n_{i}}(1+Y)^{-v_{i}}=0,\qquad y\in(0,\infty)^{d}.\label{eq:one_firstd_class}
\end{equation}

Now partition $C$ by the relation 
\[
i\approx i'\quad\Longleftrightarrow\quad\alpha_{J}^{(i)}-\alpha_{J}^{(i')}\in\mathbb{Z}.
\]
Because all indices in $C$ already have the same first $d$ coordinates
modulo $\mathbb{Z}$, this is equivalent to $\alpha^{(i)}-\alpha^{(i')}\in\mathbb{Z}^{J}$.

For $i\in C$ define 
\[
\lambda_{i}=v_{i}-|n_{i}|.
\]
Since $u_{i}=u_{0}+n_{i}$, we have 
\[
\lambda_{i}=\alpha_{+}^{(i)}-\sum_{j=1}^{d}n_{ij}=\alpha_{J}^{(i)}+\sum_{j=1}^{d}\bigl((u_{0})_{j}+1\bigr),
\]
so 
\[
\lambda_{i}-\lambda_{i'}=\alpha_{J}^{(i)}-\alpha_{J}^{(i')}.
\]
Hence the partition by $\approx$ is also the partition by $\lambda_{i}$
modulo $\mathbb{Z}$.

Let $\Delta_{d-1}^{\circ}=\{x\in\mathbb{R}_{>0}^{d}:x_{1}+\cdots+x_{d}=1\}$.
For $x\in\Delta_{d-1}^{\circ}$ and $s\in(0,1)$ set $y=x/s$. Then
$Y=1/s$, and \eqref{eq:one_firstd_class} becomes 
\begin{equation}
\sum_{i\in C}b_{i}\,x^{n_{i}}s^{\lambda_{i}}(1+s)^{-v_{i}}=0,\qquad x\in\Delta_{d-1}^{\circ},\ 0<s<1.\label{eq:s_identity_congruence}
\end{equation}

Fix one $\approx$-class $D\subseteq C$ and let 
\[
\lambda_{D}=\min\{\lambda_{i}:i\in D\}.
\]
Then for each $i\in D$ there exists $k_{i}\in\mathbb{N}_{0}$ such
that 
\[
\lambda_{i}=\lambda_{D}+k_{i}.
\]
For $|s|<1$ we have the convergent binomial expansion 
\[
(1+s)^{-v_{i}}=\sum_{m=0}^{\infty}\frac{(-1)^{m}(v_{i})_{m}}{m!}\,s^{m}.
\]
Substituting this into \eqref{eq:s_identity_congruence} yields, for
each fixed $x\in\Delta_{d-1}^{\circ}$, 
\[
0=\sum_{D}s^{\lambda_{D}}\sum_{n=0}^{\infty}A_{D,n}(x)s^{n},\qquad0<s<1,
\]
where 
\[
A_{D,n}(x)=\sum_{i\in D}b_{i}\,x^{n_{i}}\,\frac{(-1)^{n-k_{i}}(v_{i})_{n-k_{i}}}{(n-k_{i})!},
\]
with the convention that the summand is $0$ when $n<k_{i}$.

We now use a one-dimensional uniqueness argument. Suppose there exists
a pair $(D,n)$ with $A_{D,n}(x)\neq0$. Set
\[
E_{x}=\{\lambda_{D}+n:A_{D,n}(x)\neq0\}.
\]
Because there are only finitely many $\approx$-classes $D$, for every
$B>0$ there are only finitely many pairs $(D,n)$ satisfying $\lambda_{D}+n\le B$.
Hence $E_{x}$ is a nonempty subset of $(0,\infty)$ with a smallest element;
write $\eta=\min E_{x}$, and choose $(D_{*},n_{*})$ such that $\eta=\lambda_{D_{*}}+n_{*}$.
Because the numbers $\lambda_{D}$ are distinct modulo integers, the
representation $\eta=\lambda_{D}+n$ is unique. Multiplying the above
identity by $s^{-\eta}$ gives
\[
0=A_{D_{*},n_{*}}(x)+\sum_{(D,n):\,\lambda_{D}+n>\eta}A_{D,n}(x)s^{\lambda_{D}+n-\eta},
\qquad 0<s<1.
\]
Every exponent in the remainder sum is strictly positive, so letting
$s\downarrow0$ yields $A_{D_{*},n_{*}}(x)=0$, a contradiction. Therefore 
\[
A_{D,n}(x)=0\qquad\text{for every \ensuremath{D}, every \ensuremath{n\ge0}, and every \ensuremath{x\in\Delta_{d-1}^{\circ}}.}
\]
It follows that for each $\approx$-class $D$, 
\[
\sum_{i\in D}b_{i}\,x^{n_{i}}s^{\lambda_{i}}(1+s)^{-v_{i}}=0,\qquad x\in\Delta_{d-1}^{\circ},\ 0<s<1.
\]
Undoing the change of variables $y=x/s$, we obtain 
\[
\sum_{i\in D}b_{i}\,y^{n_{i}}(1+Y)^{-v_{i}}=0,\qquad y\in(0,\infty)^{d}\text{ with }Y>1.
\]
Thus the corresponding partial sum vanishes on the nonempty open set
$\{y>0:Y>1\}$. By the same identity theorem, it vanishes on all of
$\mathbb{R}_{>0}^{d}$. Multiplying back by the common factor $y^{u_{0}}$
shows that 
\[
\sum_{i\in D}c_{i}\,h_{\alpha^{(i)}}(y)=0,\qquad y\in(0,\infty)^{d}.
\]
Since $C$ was arbitrary, the same conclusion holds for every congruence
class. 
\end{proof}
The next lemma is the key combinatorial input. It concerns polynomial
identities on the simplex and is independent of the Dirichlet kernel
normalizing constants. The proof is elementary, but its degree-elevation
step is closely related to the Bernstein--Bézier representation and
degree-raising formula on triangles; see, e.g., \citet[Thms.~2.4
and~2.39]{LaiSchumaker2007}. 
\begin{lem}
\label{lem:simplex_poly_sign} Let $n^{(1)},\ldots,n^{(M)}\in\mathbb{N}_{0}^{J}$
be distinct and let $d_{1},\ldots,d_{M}\in\mathbb{R}$ be nonzero.
Assume 
\[
\sum_{i=1}^{M}d_{i}\,x^{n^{(i)}}=0,\qquad x\in\Delta_{J-1}^{\circ},\qquad x^{n}=\prod_{j=1}^{J}x_{j}^{n_{j}}.
\]
Then 
\[
\max\bigl(|I_{+}|,|I_{-}|\bigr)\ge J,\qquad I_{+}=\{i:d_{i}>0\},\ \ I_{-}=\{i:d_{i}<0\}.
\]
\end{lem}

\begin{proof}
Let 
\[
N=\max_{1\le i\le M}|n^{(i)}|,\qquad|n|=n_{1}+\cdots+n_{J}.
\]
Because $x_{1}+\cdots+x_{J}=1$ on the simplex, each monomial may
be degree-elevated to level $N$: 
\[
x^{n^{(i)}}=x^{n^{(i)}}(x_{1}+\cdots+x_{J})^{N-|n^{(i)}|}=\sum_{\substack{\beta\in\mathbb{N}_{0}^{J}\\
|\beta|=N-|n^{(i)}|
}
}\binom{N-|n^{(i)}|}{\beta}\,x^{n^{(i)}+\beta}.
\]
Hence the assumed identity becomes 
\[
0=\sum_{|u|=N}A_{u}\,x^{u},\qquad A_{u}=\sum_{i=1}^{M}d_{i}\,\binom{N-|n^{(i)}|}{u-n^{(i)}},
\]
where the multinomial coefficient is interpreted as $0$ if $u\not\ge n^{(i)}$
coordinatewise.

For each $u\in\mathbb{N}_{0}^{J}$ with $|u|=N$, the kernel $f_{u+\mathbf{1}}$
is a positive constant multiple of $x^{u}$, and all parameters $u+\mathbf{1}$
lie in the fixed-total slice $\Theta_{N+J}=\{\alpha:\alpha_{+}=N+J\}$.
By Corollary~\ref{cor:fixedA_linind}, these kernels are linearly
independent. Therefore the coefficients $A_{u}$ must vanish individually:
\begin{equation}
A_{u}=0\qquad\text{for every }u\in\mathbb{N}_{0}^{J}\text{ with }|u|=N.\label{eq:Au_zero}
\end{equation}

If all exponents $n^{(i)}$ had total degree $N$, then each function
$x^{n^{(i)}}$ would be a positive constant multiple of the Dirichlet
kernel $f_{n^{(i)}+\mathbf{1}}$, and all parameters $n^{(i)}+\mathbf{1}$
would lie in the same fixed-total slice $\Theta_{N+J}$. Corollary~\ref{cor:fixedA_linind}
would then force all coefficients $d_{i}$ to vanish, contrary to the
assumption of a nontrivial identity. Hence at least one exponent has
total degree strictly less than $N$.

Choose $i_{*}$ so that 
\[
|n_{*}|=\min_{1\le i\le M}|n^{(i)}|,\qquad n_{*}=n^{(i_{*})}.
\]
Then $|n_{*}|<N$. Set $\sigma=\operatorname{sgn}(d_{i_{*}})\in\{+1,-1\}$.

For each $j\in\{1,\ldots,J\}$ define the degree-$N$ exponent 
\[
\nu^{(j)}=n_{*}+(N-|n_{*}|)e_{j}.
\]
Because $N-|n_{*}|>0$, the vectors $\nu^{(1)},\dots,\nu^{(J)}$ are
distinct and satisfy $|\nu^{(j)}|=N$. Hence \eqref{eq:Au_zero} gives,
for each $j$, 
\begin{equation}
0=A_{\nu^{(j)}}=d_{i_{*}}+\sum_{i\ne i_{*}}d_{i}\,\binom{N-|n^{(i)}|}{\nu^{(j)}-n^{(i)}}.\label{eq:vertex_coeff_eq}
\end{equation}
The coefficient of $d_{i_{*}}$ is $1$ because $\nu^{(j)}-n_{*}=(N-|n_{*}|)e_{j}$.
Moreover every multinomial coefficient in \eqref{eq:vertex_coeff_eq}
is nonnegative. Hence terms with sign $\sigma$ contribute with sign
$\sigma$, and for the sum to vanish there must be at least one index
of sign $-\sigma$ contributing to each equation \eqref{eq:vertex_coeff_eq}.

We claim that any fixed index $i\ne i_{*}$ can contribute to at most
one of the $J$ equations \eqref{eq:vertex_coeff_eq}. Indeed, if
$n^{(i)}\le\nu^{(j)}$ and $n^{(i)}\le\nu^{(k)}$ for two distinct
indices $j\ne k$, then for every coordinate $\ell\notin\{j,k\}$
we have 
\[
n_{\ell}^{(i)}\le\nu_{\ell}^{(j)}=n_{*,\ell}.
\]
For $\ell=j$, the inequality $n^{(i)}\le\nu^{(k)}$ gives 
\[
n_{j}^{(i)}\le\nu_{j}^{(k)}=n_{*,j},
\]
and similarly $n_{k}^{(i)}\le n_{*,k}$. Thus $n^{(i)}\le n_{*}$
coordinatewise. Since all coordinates are nonnegative integers and
$n^{(i)}\neq n_{*}$, this would imply $|n^{(i)}|<|n_{*}|$, contradicting
the minimal choice of $|n_{*}|$. The claim follows.

Therefore the $J$ distinct identities \eqref{eq:vertex_coeff_eq}
require at least $J$ distinct indices of sign $-\sigma$. Hence either
$|I_{+}|\ge J$ or $|I_{-}|\ge J$, i.e. 
\[
\max\bigl(|I_{+}|,|I_{-}|\bigr)\ge J.
\]
\end{proof}
We can now transfer the preceding combinatorial statement back to
Dirichlet kernels. 
\begin{lem}
\label{lem:sign_lower_bound} Let $\alpha^{(1)},\ldots,\alpha^{(M)}\in(0,\infty)^{J}$
be distinct and let $c_{1},\ldots,c_{M}\in\mathbb{R}$ not all zero.
If 
\[
\sum_{i=1}^{M}c_{i}\,f_{\alpha^{(i)}}(x)=0,\qquad x\in\Delta_{J-1}^{\circ},
\]
then 
\[
\max\left(|I_{+}|,|I_{-}|\right)\ge J,\qquad I_{+}=\{i:c_{i}>0\},\ I_{-}=\{i:c_{i}<0\}.
\]
\end{lem}

\begin{proof}
Using \eqref{eq:f_vs_h}, the identity among Dirichlet kernels is
equivalent to 
\[
\sum_{i=1}^{M}c_{i}\,h_{\alpha^{(i)}}(y)=0,\qquad y\in\mathbb{R}_{>0}^{d}.
\]
By Lemma~\ref{lem:full_congruence}, this identity splits according
to the congruence class of the parameter modulo $\mathbb{Z}^{J}$.
Since not all $c_{i}$ vanish, there exists at least one congruence
class $D$ for which 
\begin{equation}
\sum_{i\in D}c_{i}\,h_{\alpha^{(i)}}(y)=0,\qquad y\in\mathbb{R}_{>0}^{d},\label{eq:one_full_congruence_class}
\end{equation}
and at least one coefficient in this sum is nonzero.

Fix such a class $D$. By \eqref{eq:f_vs_h}, for $x=x(y)$ we have
\[
\sum_{i\in D}c_{i}f_{\alpha^{(i)}}(x(y))=x_{J}(y)^{-J}\sum_{i\in D}c_{i}h_{\alpha^{(i)}}(y)=0.
\]
Since $y\mapsto x(y)$ is a bijection from $\mathbb{R}_{>0}^{d}$
onto $\Delta_{J-1}^{\circ}$, it follows that 
\[
\sum_{i\in D}c_{i}f_{\alpha^{(i)}}(x)=0,\qquad x\in\Delta_{J-1}^{\circ}.
\]

Because all parameters in $D$ are congruent modulo $\mathbb{Z}^{J}$,
there exists a unique vector $r=(r_{1},\ldots,r_{J})\in(0,1]^{J}$
and distinct vectors $n^{(i)}\in\mathbb{N}_{0}^{J}$ ($i\in D$) such
that 
\[
\alpha^{(i)}=r+n^{(i)},\qquad i\in D.
\]
Then 
\[
f_{\alpha^{(i)}}(x)=\frac{\Gamma(r_{+}+|n^{(i)}|)}{\prod_{j=1}^{J}\Gamma(r_{j}+n_{j}^{(i)})}\,x^{r-\mathbf{1}}\,x^{n^{(i)}},\qquad x^{r-\mathbf{1}}=\prod_{j=1}^{J}x_{j}^{r_{j}-1}.
\]
Hence the last displayed identity is equivalent to 
\[
\sum_{i\in D}d_{i}\,x^{n^{(i)}}=0,\qquad x\in\Delta_{J-1}^{\circ},
\]
where 
\[
d_{i}=c_{i}\,\frac{\Gamma(r_{+}+|n^{(i)}|)}{\prod_{j=1}^{J}\Gamma(r_{j}+n_{j}^{(i)})}.
\]
All prefactors are strictly positive, so 
\[
\operatorname{sgn}(d_{i})=\operatorname{sgn}(c_{i}),\qquad i\in D.
\]
Discard any indices with $d_{i}=0$, which does not change the identity
or the sign sets. Applying Lemma~\ref{lem:simplex_poly_sign} to
the remaining polynomial identity yields 
\[
\max\bigl(|\{i\in D:d_{i}>0\}|,\ |\{i\in D:d_{i}<0\}|\bigr)\ge J.
\]
Since signs are preserved from $d_{i}$ to $c_{i}$, the same bound
holds for the original coefficients $c_{i}$: 
\[
\max\left(|I_{+}|,|I_{-}|\right)\ge J.
\]
\end{proof}
The sign lower bound immediately yields the promised identifiability
result for mixtures with fewer than $J$ atoms. 
\begin{thm}
\label{thm:K_less_than_J_ident} For each integer $K<J$, let 
\[
\mathcal{F}_{K}=\left\{ m_{G}:G\text{ is a finite discrete mixing measure on }\mathbb{R}_{>0}^{J}\text{ with at most }K\text{ atoms}\right\} .
\]
Then $\mathcal{F}_{K}$ is identifiable. In particular, $\mathcal{F}_{J-1}$
is identifiable. 
\end{thm}

\begin{proof}
Write 
\[
G=\sum_{k=1}^{K_{1}}\pi_{k}\delta_{\alpha^{(k)}},\qquad G'=\sum_{\ell=1}^{K_{2}}\rho_{\ell}\delta_{\beta^{(\ell)}},
\]
with $K_{1},K_{2}\le K$, $\pi_{k}>0$, $\rho_{\ell}>0$, and distinct
atoms in each measure. As in the previous proof, continuity upgrades
equality almost everywhere to equality for every $x\in\Delta_{J-1}^{\circ}$:
\[
m_{G}(x)=m_{G'}(x),\qquad x\in\Delta_{J-1}^{\circ}.
\]
Bringing all terms to one side and collecting equal atoms gives a
representation 
\[
\sum_{i=1}^{M}c_{i}f_{\gamma^{(i)}}(x)=0,\qquad x\in\Delta_{J-1}^{\circ},
\]
where the $\gamma^{(i)}$ are distinct and the coefficients $c_{i}$
are the singleton-mass differences 
\[
c_{i}=G(\{\gamma^{(i)}\})-G'(\{\gamma^{(i)}\}).
\]
Discard any zero coefficients. If all coefficients vanish then $G=G'$
and there is nothing to prove, so suppose we are in the nontrivial
case.

Define 
\[
I_{+}=\{i:c_{i}>0\},\qquad I_{-}=\{i:c_{i}<0\}.
\]
If $i\in I_{+}$ then $\gamma^{(i)}$ is an atom of $G$, so $|I_{+}|\le K_{1}\le K$.
Similarly $|I_{-}|\le K_{2}\le K$. Hence 
\[
\max(|I_{+}|,|I_{-}|)\le K<J.
\]
But Lemma~\ref{lem:sign_lower_bound} applies to the nontrivial relation
above and yields 
\[
\max(|I_{+}|,|I_{-}|)\ge J,
\]
a contradiction. Therefore the nontrivial case is impossible, so all
coefficients must in fact be zero. Equivalently, 
\[
G(\{\gamma^{(i)}\})=G'(\{\gamma^{(i)}\})\qquad\text{for all }i.
\]
Since both measures are finite and supported on the same finite set
of atoms, this implies $G=G'$. 
\end{proof}

\section{Implications}

\label{sec:Implications}

Several commonly used models are obtained from the Dirichlet family
either by embedding it as a special case or by integrating the Dirichlet
density against a fixed kernel. The non-identifiability results of the
previous sections therefore propagate directly to those settings.

Following \citet{ConnorMosimann1969} and \citet{wong2010gd}, the
generalized Dirichlet density on $\Delta_{J-1}^{\circ}$ with parameters
$a_{1},\dots,a_{J-1}>0$ and $b_{1},\dots,b_{J-1}>0$ is 
\[
g_{a,b}(x)=\prod_{j=1}^{J-1}\frac{\Gamma(a_{j}+b_{j})}{\Gamma(a_{j})\Gamma(b_{j})}x_{j}^{a_{j}-1}\left(1-\sum_{\ell=1}^{j}x_{\ell}\right)^{\gamma_{j}},
\]
where 
\[
\gamma_{j}=b_{j}-a_{j+1}-b_{j+1}\qquad(1\le j\le J-2),\qquad\gamma_{J-1}=b_{J-1}-1.
\]

\begin{prop}
\label{prop:GD_nonid} The class of unrestricted finite mixtures of
generalized Dirichlet distributions on $\Delta_{J-1}^{\circ}$ is
not identifiable. 
\end{prop}

\begin{proof}
Fix $\alpha\in\mathbb{R}_{>0}^{J}$. Set 
\[
a_{j}=\alpha_{j}\qquad\text{and}\qquad b_{j}=\sum_{\ell=j+1}^{J}\alpha_{\ell},\qquad j=1,\dots,J-1.
\]
Then for $j=1,\dots,J-2$, 
\[
\gamma_{j}=b_{j}-a_{j+1}-b_{j+1}=\sum_{\ell=j+1}^{J}\alpha_{\ell}-\alpha_{j+1}-\sum_{\ell=j+2}^{J}\alpha_{\ell}=0,
\]
while 
\[
\gamma_{J-1}=b_{J-1}-1=\alpha_{J}-1.
\]
Moreover, the normalizing constant telescopes: 
\[
\prod_{j=1}^{J-1}\frac{\Gamma(a_{j}+b_{j})}{\Gamma(a_{j})\Gamma(b_{j})}=\prod_{j=1}^{J-1}\frac{\Gamma(\sum_{\ell=j}^{J}\alpha_{\ell})}{\Gamma(\alpha_{j})\Gamma(\sum_{\ell=j+1}^{J}\alpha_{\ell})}=\frac{\Gamma(\alpha_{+})}{\prod_{j=1}^{J}\Gamma(\alpha_{j})}.
\]
Hence 
\[
g_{a,b}(x)=\frac{\Gamma(\alpha_{+})}{\prod_{j=1}^{J}\Gamma(\alpha_{j})}\prod_{j=1}^{J-1}x_{j}^{\alpha_{j}-1}x_{J}^{\alpha_{J}-1}=f_{\alpha}(x).
\]
Thus every Dirichlet density is also a generalized Dirichlet density.
Therefore every finite Dirichlet mixture is also a finite generalized
Dirichlet mixture. If unrestricted finite generalized Dirichlet mixtures
were identifiable, then Corollary~\ref{cor:nonid} would be impossible. 
\end{proof}
Dirichlet-multinomial models arise by integrating multinomial sampling
against a Dirichlet prior; see \citet{bouguila2008gdm} and \citet{holmes2012dmm}.
The shift identity from Theorem~\ref{thm:shift} therefore survives
this marginalization step. 
\begin{prop}
\label{prop:DM_nonid} Fix $n\in\mathbb{N}$ and write 
\[
\mathcal{X}_{n}=\{x\in\mathbb{N}_{0}^{J}:x_{1}+\cdots+x_{J}=n\}.
\]
For $\alpha\in\mathbb{R}_{>0}^{J}$, let 
\[
p_{n,\alpha}(x)=\int_{\Delta_{J-1}^{\circ}}\frac{n!}{x_{1}!\cdots x_{J}!}\prod_{j=1}^{J}p_{j}^{x_{j}}f_{\alpha}(p)\,dp,\qquad x\in\mathcal{X}_{n},
\]
be the Dirichlet-multinomial kernel. Then 
\[
p_{n,\alpha}(x)=\sum_{j=1}^{J}\frac{\alpha_{j}}{\alpha_{+}}\,p_{n,\alpha+e_{j}}(x),\qquad x\in\mathcal{X}_{n}.
\]
Consequently, the class of unrestricted finite Dirichlet-multinomial
mixtures with fixed total count $n$ is not identifiable. 
\end{prop}

\begin{proof}
The displayed identity is obtained by multiplying the shift identity
of Theorem~\ref{thm:shift} by the multinomial kernel $\tfrac{n!}{x_{1}!\cdots x_{J}!}\prod_{j=1}^{J}p_{j}^{x_{j}}$
and integrating over $\Delta_{J-1}^{\circ}$. The non-identifiability
conclusion then follows exactly as in Corollary~\ref{cor:nonid},
using the one-atom measure $\delta_{\alpha}$ and the $J$-atom measure
$\sum_{j=1}^{J}(\alpha_{j}/\alpha_{+})\,\delta_{\alpha+e_{j}}$. If
the totals are observed and conditioned upon, the same identity holds
for each fixed total $n$, so the same argument applies pointwise
in $n$. 
\end{proof}
For latent Dirichlet allocation, \citet{blei2003lda} consider a vocabulary
of size $V$, a topic matrix $\beta=(\beta_{kv})\in[0,1]^{K\times V}$
with $\sum_{v=1}^{V}\beta_{kv}=1$ for each $k$, and documents $w=(w_{1},\dots,w_{N})\in\{1,\dots,V\}^{N}$.
For fixed $\beta$, the marginal likelihood of a document $w$ under
the Dirichlet prior $\mathrm{Dir}(\alpha)$ can be written as 
\[
q_{\alpha,\beta}(w)=\int_{\Delta_{K-1}^{\circ}}\prod_{n=1}^{N}\Bigl(\sum_{k=1}^{K}\theta_{k}\beta_{k,w_{n}}\Bigr)f_{\alpha}(\theta)\,d\theta.
\]
Thus the same integration argument yields the following analogue of
Proposition~\ref{prop:DM_nonid}. 
\begin{prop}
\label{prop:LDA_nonid} Fix a topic matrix $\beta$. Then for every
$\alpha\in\mathbb{R}_{>0}^{K}$ and every document $w$, 
\[
q_{\alpha,\beta}(w)=\sum_{k=1}^{K}\frac{\alpha_{k}}{\alpha_{+}}\,q_{\alpha+e_{k},\beta}(w).
\]
Consequently, the unrestricted finite mixture class generated by the LDA marginal
kernels $\{q_{\alpha,\beta}:\alpha\in\mathbb{R}_{>0}^{K}\}$ is not
identifiable. More generally, the same conclusion holds for any parameter
set $\Theta\subseteq\mathbb{R}_{>0}^{K}$ that contains some $\alpha$ together
with all shifts $\alpha+e_{k}$, $k=1,\dots,K$. 
\end{prop}

\begin{proof}
For fixed $\beta$ and $w$, the integrand $\prod_{n=1}^{N}(\sum_{k=1}^{K}\theta_{k}\beta_{k,w_{n}})$
depends on $\alpha$ only through the Dirichlet density $f_{\alpha}(\theta)$.
Integrating the shift identity of Theorem~\ref{thm:shift} over the
topic simplex therefore gives the stated equality. The non-identifiability
of the unrestricted finite mixture class is then proved exactly as in Corollary~\ref{cor:nonid};
the same argument also applies to any parameter set containing one full
unit-shift orbit $\{\alpha+e_{k}:1\le k\le K\}$ together with $\alpha$. 
\end{proof}
Following \citet{fan2013bl}, the Beta-Liouville density on the simplex
$\Delta_{J-1}^{\circ}$, with parameters $a_{1},\dots,a_{J-1}>0$,
$a>0$, and $b>0$, is 
\[
b_{a_{1},\dots,a_{J-1},a,b}(x)=\frac{\Gamma(a+b)\Gamma(\sum_{j=1}^{J-1}a_{j})}{\Gamma(a)\Gamma(b)\prod_{j=1}^{J-1}\Gamma(a_{j})}\left(\prod_{j=1}^{J-1}x_{j}^{a_{j}-1}\right)\left(\sum_{j=1}^{J-1}x_{j}\right)^{a-\sum_{j=1}^{J-1}a_{j}}x_{J}^{b-1}.
\]
Its unrestricted mixture class inherits the same non-identifiability
by inclusion of the Dirichlet subfamily. 
\begin{prop}
\label{prop:BL_nonid} The class of unrestricted finite Beta-Liouville
mixtures is not identifiable. 
\end{prop}

\begin{proof}
If $a=\sum_{j=1}^{J-1}a_{j}$, then the middle factor is identically
$1$, and the normalizing constant reduces to 
\[
\frac{\Gamma(\sum_{j=1}^{J-1}a_{j}+b)}{\prod_{j=1}^{J-1}\Gamma(a_{j})\Gamma(b)}.
\]
Hence 
\[
b_{a_{1},\dots,a_{J-1},a,b}(x)=\frac{\Gamma(\sum_{j=1}^{J-1}a_{j}+b)}{\prod_{j=1}^{J-1}\Gamma(a_{j})\Gamma(b)}\prod_{j=1}^{J-1}x_{j}^{a_{j}-1}x_{J}^{b-1},
\]
which is exactly the Dirichlet density with parameter $(a_{1},\dots,a_{J-1},b)$.
Therefore the unrestricted Beta-Liouville family contains the unrestricted
Dirichlet family as a subfamily, and Corollary~\ref{cor:nonid} implies
the claim. 
\end{proof}
Following \citet{hu2019ibl}, the inverted Beta-Liouville density
on $\mathbb{R}_{>0}^{J-1}$, with parameters $a_{1},\dots,a_{J-1}>0$,
$a>0$, $b>0$, and $\lambda>0$, is 
\[
\tilde{b}_{a_{1},\dots,a_{J-1},a,b,\lambda}(y)=\frac{\Gamma(\sum_{j=1}^{J-1}a_{j})\Gamma(a+b)}{\Gamma(a)\Gamma(b)\prod_{j=1}^{J-1}\Gamma(a_{j})}\lambda^{b}\left(\prod_{j=1}^{J-1}y_{j}^{a_{j}-1}\right)\left(\sum_{j=1}^{J-1}y_{j}\right)^{a-\sum_{j=1}^{J-1}a_{j}}\left(\lambda+\sum_{j=1}^{J-1}y_{j}\right)^{-(a+b)}.
\]
This family contains the inverted Dirichlet family as a special case,
so Remark~\ref{rem:inverted_transfer} yields the corresponding consequence
on the positive orthant. 
\begin{prop}
\label{prop:IBL_nonid} The class of unrestricted finite inverted
Beta-Liouville mixtures is not identifiable. 
\end{prop}

\begin{proof}
If $a=\sum_{j=1}^{J-1}a_{j}$ and $\lambda=1$, then 
\[
\tilde{b}_{a_{1},\dots,a_{J-1},a,b,1}(y)=\frac{\Gamma(\sum_{j=1}^{J-1}a_{j}+b)}{\Gamma(b)\prod_{j=1}^{J-1}\Gamma(a_{j})}\frac{\prod_{j=1}^{J-1}y_{j}^{a_{j}-1}}{(1+\sum_{j=1}^{J-1}y_{j})^{\sum_{j=1}^{J-1}a_{j}+b}}.
\]
This is precisely the inverted Dirichlet density $h_{\gamma}(y)$
from \eqref{eq:transported_kernel}, with 
\[
\gamma=(a_{1},\dots,a_{J-1},b).
\]
By Remark~\ref{rem:inverted_transfer}, non-identifiability of unrestricted
finite Dirichlet mixtures transfers to unrestricted finite inverted
Dirichlet mixtures. Since the inverted Beta-Liouville family contains
that inverted Dirichlet subfamily, unrestricted finite inverted Beta-Liouville
mixtures are likewise non-identifiable. 
\end{proof}
\begin{rem}
The propositions above are all one-way consequences of the Dirichlet
results proved in this paper. They establish non-identifiability for
the unrestricted generalized Dirichlet, Dirichlet-multinomial, fixed-topic-matrix LDA,
Beta-Liouville, and inverted Beta-Liouville classes because these
models either contain the Dirichlet (or inverted Dirichlet) family
as a subfamily or depend linearly on the Dirichlet density. By contrast,
the positive identifiability results of Theorems~\ref{thm:fixedA},
\ref{thm:linind_FD}, and \ref{thm:K_less_than_J_ident} do not automatically
extend to the full parameter spaces of the generalized Dirichlet,
Beta-Liouville, or inverted Beta-Liouville families. A natural direction
for future work is to identify structured subclasses of these richer
families whose induced Dirichlet atoms remain inside the identifiable
regions described in Section~\ref{sec:Identifiable-restrictions},
or to determine whether analogues of Theorem~\ref{thm:K_less_than_J_ident}
hold when the number of mixture components is small. 
\end{rem}

\bibliographystyle{plainnat}
\bibliography{dirichlet_bib}

\begin{thebibliography}{25}
\providecommand{\natexlab}[1]{#1}
\providecommand{\url}[1]{\texttt{#1}}
\expandafter\ifx\csname urlstyle\endcsname\relax
  \providecommand{\doi}[1]{doi: #1}\else
  \providecommand{\doi}{doi: \begingroup \urlstyle{rm}\Url}\fi

\bibitem[Ahmad and Al-Hussaini(1982)]{AhmadAlHussaini1982}
Khalaf~E. Ahmad and Essam~K. Al-Hussaini.
\newblock Remarks on the non-identifiability of mixtures of distributions.
\newblock \emph{Annals of the Institute of Statistical Mathematics}, 34\penalty0 (1):\penalty0 543--544, 1982.

\bibitem[Blei et~al.(2003)Blei, Ng, and Jordan]{blei2003lda}
David~M. Blei, Andrew~Y. Ng, and Michael~I. Jordan.
\newblock Latent {Dirichlet} allocation.
\newblock \emph{Journal of Machine Learning Research}, 3:\penalty0 993--1022, 2003.

\bibitem[Bouguila(2008)]{bouguila2008gdm}
Nizar Bouguila.
\newblock Clustering of count data using generalized {Dirichlet} multinomial distributions.
\newblock \emph{IEEE Transactions on Knowledge and Data Engineering}, 20\penalty0 (4):\penalty0 462--474, 2008.

\bibitem[Bouguila et~al.(2004)Bouguila, Ziou, and Vaillancourt]{bouguila2004unsupervised}
Nizar Bouguila, Djemel Ziou, and Jean Vaillancourt.
\newblock Unsupervised learning of a finite mixture model based on the {Dirichlet} distribution and its application.
\newblock \emph{IEEE Transactions on Image Processing}, 13\penalty0 (11):\penalty0 1533--1543, 2004.

\bibitem[Chen(2023)]{chen2023statistical}
Jiahua Chen.
\newblock \emph{Statistical Inference Under Mixture Models}.
\newblock Springer, Singapore, 2023.

\bibitem[Connor and Mosimann(1969)]{ConnorMosimann1969}
Richard~J. Connor and James~E. Mosimann.
\newblock Concepts of independence for proportions with a generalization of the {Dirichlet} distribution.
\newblock \emph{Journal of the American Statistical Association}, 64\penalty0 (325):\penalty0 194--206, 1969.

\bibitem[Fan and Bouguila(2013)]{fan2013bl}
Wentao Fan and Nizar Bouguila.
\newblock Learning finite {Beta-Liouville} mixture models via variational {Bayes} for proportional data clustering.
\newblock In Francesca Rossi, editor, \emph{Proceedings of the Twenty-Third International Joint Conference on Artificial Intelligence ({IJCAI} 2013)}, pages 1323--1329, Beijing, China, 2013. IJCAI/AAAI.

\bibitem[Fan et~al.(2012)Fan, Bouguila, and Ziou]{fan2012variational}
Wentao Fan, Nizar Bouguila, and Djemel Ziou.
\newblock Variational learning for finite {Dirichlet} mixture models and applications.
\newblock \emph{IEEE Transactions on Neural Networks and Learning Systems}, 23\penalty0 (5):\penalty0 762--774, 2012.

\bibitem[Fan et~al.(2016)Fan, Sallay, Bouguila, and Bourouis]{fan2016variational}
Wentao Fan, Hassen Sallay, Nizar Bouguila, and Sami Bourouis.
\newblock Variational learning of hierarchical infinite generalized {Dirichlet} mixture models and applications.
\newblock \emph{Soft Computing}, 20\penalty0 (3):\penalty0 979--990, 2016.

\bibitem[Graham et~al.(1994)Graham, Knuth, and Patashnik]{GrahamKnuthPatashnik1994}
Ronald~L. Graham, Donald~E. Knuth, and Oren Patashnik.
\newblock \emph{Concrete Mathematics: A Foundation for Computer Science}.
\newblock Addison-Wesley, Reading, MA, 2nd edition, 1994.

\bibitem[Gr{\"u}n et~al.(2012)Gr{\"u}n, Kosmidis, and Zeileis]{grun2012extended}
Bettina Gr{\"u}n, Ioannis Kosmidis, and Achim Zeileis.
\newblock Extended {Beta} regression in {R}: shaken, stirred, mixed, and partitioned.
\newblock \emph{Journal of Statistical Software}, 48\penalty0 (11):\penalty0 1--25, 2012.

\bibitem[Holmes et~al.(2012)Holmes, Harris, and Quince]{holmes2012dmm}
Ian Holmes, Kelly Harris, and Christopher Quince.
\newblock {Dirichlet} multinomial mixtures: generative models for microbial metagenomics.
\newblock \emph{PLOS ONE}, 7\penalty0 (2):\penalty0 e30126, 2012.

\bibitem[Houseman et~al.(2008)Houseman, Christensen, Yeh, Marsit, Karagas, Wrensch, Nelson, Wiemels, Zheng, Wiencke, and Kelsey]{houseman2008model}
E.~Andres Houseman, Brock~C. Christensen, Ru-Fang Yeh, Carmen~J. Marsit, Margaret~R. Karagas, Margaret Wrensch, Heather~H. Nelson, Joseph Wiemels, Shichun Zheng, John~K. Wiencke, and Karl~T. Kelsey.
\newblock Model-based clustering of {DNA} methylation array data: a recursive-partitioning algorithm for high-dimensional data arising as a mixture of {Beta} distributions.
\newblock \emph{BMC Bioinformatics}, 9:\penalty0 365, 2008.

\bibitem[Hu et~al.(2019)Hu, Fan, Du, and Bouguila]{hu2019ibl}
Can Hu, Wentao Fan, Ji-Xiang Du, and Nizar Bouguila.
\newblock A novel statistical approach for clustering positive data based on finite inverted {Beta-Liouville} mixture models.
\newblock \emph{Neurocomputing}, 333:\penalty0 110--123, 2019.

\bibitem[Ji et~al.(2005)Ji, Wu, Liu, Wang, and Coombes]{ji2005applications}
Yuan Ji, Chunlei Wu, Ping Liu, Jing Wang, and Kevin~R. Coombes.
\newblock Applications of {Beta}-mixture models in bioinformatics.
\newblock \emph{Bioinformatics}, 21\penalty0 (9):\penalty0 2118--2122, 2005.

\bibitem[Kotz et~al.(2019)Kotz, Balakrishnan, and Johnson]{kotz2019continuous}
Samuel Kotz, N.~Balakrishnan, and Norman~L. Johnson.
\newblock \emph{Continuous Multivariate Distributions, Volume 1: Models and Applications}.
\newblock Wiley, Hoboken, NJ, 2nd edition, 2019.

\bibitem[Krantz and Parks(2002)]{KrantzParks2002}
Steven~G. Krantz and Harold~R. Parks.
\newblock \emph{A Primer of Real Analytic Functions}.
\newblock Birkh{\"a}user Boston, Boston, MA, 2nd edition, 2002.

\bibitem[Lai and Schumaker(2007)]{LaiSchumaker2007}
Ming-Jun Lai and Larry~L. Schumaker.
\newblock \emph{Spline Functions on Triangulations}.
\newblock Cambridge University Press, Cambridge, 2007.

\bibitem[Ma and Leijon(2009)]{ma2009beta}
Zhanyu Ma and Arne Leijon.
\newblock {Beta} mixture models and the application to image classification.
\newblock In \emph{Proceedings of the 2009 16th {IEEE} International Conference on Image Processing ({ICIP})}, pages 2045--2048, Piscataway, NJ, 2009. IEEE.

\bibitem[McLachlan and Basford(1988)]{McLachlanBasford1988}
Geoffrey~J. McLachlan and Kaye~E. Basford.
\newblock \emph{Mixture Models: Inference and Applications to Clustering}.
\newblock Marcel Dekker, New York, 1988.

\bibitem[Pal and Heumann(2022)]{pal2022clustering}
Samyajoy Pal and Christian Heumann.
\newblock Clustering compositional data using {Dirichlet} mixture model.
\newblock \emph{PLOS ONE}, 17\penalty0 (5):\penalty0 e0268438, 2022.

\bibitem[Simone(2022)]{Simone2022}
Rosaria Simone.
\newblock On finite mixtures of discretized {Beta} model for ordered responses.
\newblock \emph{TEST}, 31\penalty0 (3):\penalty0 828--855, 2022.

\bibitem[Teicher(1963)]{teicher1963identifiability}
Henry Teicher.
\newblock Identifiability of finite mixtures.
\newblock \emph{Annals of Mathematical Statistics}, 34\penalty0 (4):\penalty0 1265--1269, 1963.

\bibitem[Titterington et~al.(1985)Titterington, Smith, and Makov]{TitteringtonSmithMakov1985}
D.~M. Titterington, A.~F.~M. Smith, and U.~E. Makov.
\newblock \emph{Statistical Analysis of Finite Mixture Distributions}.
\newblock John Wiley \& Sons, Chichester, 1985.

\bibitem[Wong(2010)]{wong2010gd}
Tzu-Tsung Wong.
\newblock Parameter estimation for generalized {Dirichlet} distributions from the sample estimates of the first and the second moments of random variables.
\newblock \emph{Computational Statistics \& Data Analysis}, 54\penalty0 (7):\penalty0 1756--1765, 2010.

\end{thebibliography}

\end{document}